\documentclass[12pt]{amsart}
\usepackage{amssymb,latexsym,eufrak,amsmath,amscd, graphicx}
\usepackage[dvipsnames,usenames]{color}

\usepackage{amsfonts}
\usepackage[all]{xypic}
\UseTips

\renewcommand{\to}[1][]{\xrightarrow{\ #1\ }}

\newcommand{\leftidx}[3]{{\vphantom{#2}}#1#2#3}     
\newcommand{\lstar}[1]{\leftidx{^*}{#1}{}}
\newcommand{\ltightstar}[1]{\leftidx{^*}{\negthinspace#1}{}}
\newcommand{\forget}[1]{}  

\makeatletter
\renewcommand{\theenumi}{\@roman\c@enumi}
\makeatother

\renewcommand{\phi}{\varphi}
\renewcommand{\epsilon}{\varepsilon}
\renewcommand{\theta}{\vartheta}

\newcommand{\sh}{\operatorname{sh}}

\newcommand{\llbracket}{[\negthinspace[}
\newcommand{\rrbracket}{]\negthinspace]}

\def\ZZ{{\mathbf Z}}
\def\NN{{\mathbf N}}

\def\AAA{{\mathbf A}}
\def\RR{{\mathbf R}}
\def\QQ{{\mathbf Q}}

\def\cJ{\mathcal{J}}

\def\cO{\mathcal{O}}
\def\cT{\mathcal{T}}
\def\cHT{\mathcal{HT}}
\def\cTpol{\mathcal{T}^{\rm pol}}
\def\cTser{\mathcal{T}^{\rm ser}}
\def\cTiso{\mathcal{T}^{\rm iso}}
\def\cHTpol{\mathcal{HT}^{\rm pol}}
\def\cHTser{\mathcal{HT}^{\rm ser}}
\def\cHTiso{\mathcal{HT}^{\rm iso}}

\def\cU{\mathcal{U}}

\def\fra{\mathfrak{a}}
\def\frb{\mathfrak{b}}
\def\frc{\mathfrak{c}}

\def\frm{\mathfrak{m}}

\def\frp{\mathfrak{p}}

\def\N{{\mathbf N}}

\def\A{{\mathbf A}}
\def\R{{\mathbf R}}
\def\Q{{\mathbf Q}}

\def\o{\circ}

\def\.{\cdot}
\def\({\Big{(}}
\def\){\Big{)}}
\def\^{\widehat}
\def\~{\widetilde}
\def\*{{}^*\!}
\def\[{\llbracket}
\def\]{\rrbracket}

\renewcommand{\and}{ \quad \text{and} \quad }

\newcommand{\fall}{ \quad \text{for all} \ \, }

\DeclareMathOperator{\codim}{codim}

 \DeclareMathOperator{\Spec}{Spec}
 
 \DeclareMathOperator{\lct}{lct}
 
 \DeclareMathOperator{\ord}{ord}

\DeclareMathOperator{\LCT}{LCT}

\newtheorem{lemma}{Lemma}[section]
\newtheorem{theorem}[lemma]{Theorem}
\newtheorem{corollary}[lemma]{Corollary}
\newtheorem{proposition}[lemma]{Proposition}

\newtheorem{conjecture}[lemma]{Conjecture}

\theoremstyle{definition}

\newtheorem{remark}[lemma]{Remark}

\theoremstyle{remark}
\newtheorem*{remark*}{Remark}
\newtheorem*{note*}{Note}

\begin{document}

\title{Limits of log canonical thresholds}

\author[T.~de Fernex]{Tommaso de Fernex}
\address{Department of Mathematics, University of Utah,
Salt Lake City, UT 48112, USA} \email{{\tt
defernex@math.utah.edu}}

\author[M. Musta\c{t}\u{a}]{Mircea~Musta\c{t}\u{a}}
\address{Department of Mathematics, University of Michigan,
Ann Arbor, MI 48109, USA} \email{{\tt mmustata@umich.edu}}

\markboth{T.~de Fernex and M.~Musta\c t\u a}{Limits of log canonical
thresholds}

\begin{abstract}
Let $\cT_n$ denote the set of log canonical thresholds of pairs
$(X,Y)$, with $X$ a nonsingular variety of dimension $n$, and $Y$ a
nonempty closed subscheme of $X$. Using non-standard methods, we
show that every limit of a decreasing sequence in $\cT_n$ lies in
$\cT_{n-1}$, proving in this setting a conjecture of Koll\'{a}r. We
also show that $\cT_n$ is closed in $\R$; in particular, every limit
of log canonical thresholds on smooth varieties of fixed dimension
is a rational number. As a consequence of this property, we see that
in order to check Shokurov's ACC Conjecture for all $\cT_n$, it is
enough to show that $1$ is not a point of accumulation from below of
any $\cT_n$. In a different direction, we interpret the ACC
Conjecture as a semi-continuity property for log canonical
thresholds of formal power series.

\bigskip

\noindent \textsc{R\'esum\'e}. 
Dans cet article, nous analysons les ensembles $\cT_n$ de seuils log canoniques de paires $(X,Y)$, ou 
$X$ est une vari\'{e}t\'{e} lisse de dimension $n$, et $Y$ est un 
sous-sch\'ema ferm\'{e} non-vide 
de $X$. En employant des m\'{e}thodes nonstandard, nous montrons que chaque limite
d'une suite strictement d\'{e}croissante de $\cT_n$ appartienne \`a l'ensemble $\cT_{n-1}$ (ce r\'esultat 
\`a \'et\'e conjectur\'e par J. Koll\'{a}r
dans ses travaux sur le sujet). Nous montrons \'{e}galement que l'ensemble $\cT_n$
est ferm\'{e} dans $\R$, et on d\'eduit que les valeurs adh\'erentes de
l'ensemble des
seuils log canoniques des pairs $(X,Y)$ sont rationnelles, si la dimension de $X$ est major\'ee. Une autre cons\'equence de 
nos r\'esultats concerne la Conjecture ACC de Shokurov pour les $\cT_n$. En effet, nous montrons qu'elle sera une cons\'equence de
l'\'enonc\'e suivant~: pour tout $n$, la valeur $1$ ne peut 
pas $\hat{\rm e}$tre obtenue comme limite d'une suite strictement 
croissante de nombres contenus dans $\cT_n$. 
Dans une autre perspective, nous interpr\'{e}tons la Conjecture ACC comme une
propri\'{e}t\'{e} de semicontinuit\'{e} de seuils log canoniqes  des s\'{e}ries formelles. 
\end{abstract}

\thanks{2000\,\emph{Mathematics Subject Classification}.
 Primary 14B05; Secondary 03H05, 14E30.
\newline The first author was partially supported by
NSF grant DMS 0548325, and the second author
  was partially supported by  NSF grant DMS 0500127 and
  a Packard Fellowship}

\maketitle

\section{Introduction}

Let $k$ be an algebraically closed field of characteristic zero. We
consider pairs of the form $(X,Y)$, where $X$ is a smooth variety
defined over $k$ and $Y \subseteq X$ is a nonempty closed subscheme.
For every integer $n \ge 0$, we are interested in the set of all
possible log canonical thresholds in dimension $n$
$$
\cT_n(k) := \{ \lct(X,Y) \mid \text{$X$ smooth over $k$ of dimension
$n$, $\emptyset \ne Y \subseteq X$} \} \subseteq \R,
$$
where we make the convention that $\lct(X,X) = 0$.
It is well-known that $\cT_n(k) \subseteq \Q$. Note that
$\cT_0(k) = \{0\}$ and $\cT_{n-1}(k) \subseteq \cT_n(k)$ for every
$n\geq 1$.

There are two fundamental questions regarding the accumulation
points (in $\R$) of these sets.

\begin{conjecture}\label{conj_shokurov}
For every $n$, the set $\cT_n(k)$ has no points of accumulation from
below.
\end{conjecture}

\begin{conjecture}\label{conj_kollar}
For every $n\geq 1$, the set of points of accumulation from above of
$\cT_n(k)$ is equal to $\cT_{n-1}(k)$.
\end{conjecture}

In particular, the two conjectures predict that every $\cT_n(k)$ is
closed, and that the set of its accumulation points is equal to
$\cT_{n-1}(k)$. In fact, both conjectures have stronger
formulations, in which $\cT_n(k)$ is defined under weaker conditions
on the singularities of the ambient variety $X$.

Conjecture~\ref{conj_shokurov}, known as the \emph{ACC Conjecture},
was formulated by Shokurov in \cite{Sho}, where it was proved for
$n=2$ (in the more general context, alluded to in the previous
paragraph). Alexeev proved it for $n=3$ in \cite{alexeev}. The main
interest in this conjecture comes from its implications to the
Minimal Model Program, more precisely, to the Termination of Flips
Conjecture (see \cite{birkar} for a precise statement).
Conjecture~\ref{conj_kollar} above was suggested by Koll\'{a}r in
\cite{kollar}. It was shown in \cite{mckernan} that
Conjecture~\ref{conj_kollar} follows if one assumes the Minimal
Model Program and a conjecture of Alexeev--Borisov--Borisov
 on the boundedness
of $\QQ$-Fano varieties. In particular, it is known to hold (in a
more general formulation) for $n\leq 3$.

\bigskip

It is not hard to see that the set $\cT_n(k)$ is independent of $k$
(see Propositions~\ref{lem:finitely-many-lct} and 
\ref{prop:all-the-same-subset} below).
 From now on we simply write $\cT_n$
instead of $\cT_n(k)$. Our main goal is to prove
Conjecture~\ref{conj_kollar}, as well as the fact that $\cT_n$ is
closed. We state our main results in the following order.

\begin{theorem}\label{thm_main1}
For every $n$, the set $\cT_n$ is closed in $\R$.
\end{theorem}

Since $\cT_n \subseteq \Q \cap [0,n]$, this immediately implies
the following useful property.

\begin{corollary}\label{cor1}
Every limit of log canonical thresholds on smooth varieties of
bounded dimension is a rational number.
\end{corollary}

\begin{theorem}\label{thm_main2}
For every $n\geq 1$, the set of points of accumulation from above of
$\cT_n$ is equal to $\cT_{n-1}$.
\end{theorem}

We mention that there are versions of these results when instead of
arbitrary subschemes we consider only hypersurfaces. Suppose that
$\cHT_n\subseteq\cT_n$
is defined by considering only pairs $(X,Y)$, where $Y$ is
locally defined by one equation. In this case $\cHT_n=\cT_n\cap
[0,1]$, hence $\cHT_n$ is closed, too, and the set of points of
accumulation from above of $\cHT_n$ is equal to
$\cHT_{n-1}\smallsetminus\{1\}$.

Since $\cHT_n\subseteq\cT_n\subseteq n\cdot \cHT_n$, it follows that
Conjecture~\ref{conj_shokurov} holds if and only if for every $n$,
the set $\cHT_n$ has no points of accumulation from below. As a
consequence of Corollary~\ref{cor1}, we show that the conjecture can
be reduced to a special case.

\begin{corollary}\label{cor_thm_main1}
Conjecture~\ref{conj_shokurov} holds for every $n$ if and only if
the following special case holds: for every $n$, there is
$\delta_n\in (0,1)$ such that $\cHT_n\cap (\delta_n,1)=\emptyset$.
\end{corollary}

In a different direction, we investigate the ACC Conjecture using
the Zariski topology on the set of formal power series.

\begin{proposition}\label{prop:intro}
Conjecture~\ref{conj_shokurov}
holds for $n$ if and only if, assuming that $k$ is uncountable,
for every $c$ there is an integer
$N(n,c)$ such that the condition for $f$ to lie in
$$
{\mathcal R}_{n}(c):=\{ f \in k\llbracket x_1,\dots,x_n \rrbracket
\mid f(0)=0, \lct(f) \geq c \}
$$
depends only on the truncation of $f$ up to degree $N(n,c)$.
\end{proposition}

In fact, we will see that the set ${\mathcal R}_n(c)$ has the
property in the above proposition if and only if it is open inside the
maximal ideal with respect to the Zariski topology on $k\llbracket
x_1,\ldots,x_n\rrbracket$ (see \S 5). Furthermore,
Corollary~\ref{cor_thm_main1} implies that in order to prove the ACC
Conjecture for every $n$, it is enough to prove the assertion in the
proposition only for the sets ${\mathcal R}_{n}(1)$.

The main ingredient in the proof of the above theorems is given by
non-standard methods.
This approach is very natural in this context, when one wants
to encode a sequence of polynomials (or ideals) in a single object.
In our case, we start with a sequence of ideals
$\fra_m \subset k[x_1,\dots,x_n]$ whose log canonical thresholds
converge to some $c \in \RR$.
Ultrafilter constructions give non-standard extensions of our algebraic structures:
we get a field $\ltightstar{k}$ containing $k$ and a ring
$\ltightstar{(k[x_1,\dots,x_n])}$ containing $k[x_1,\dots,x_n]$.
Moreover, there is a truncation map from $\ltightstar{(k[x_1,\dots,x_n])}$
to the formal power series ring $\ltightstar{k}\llbracket x_1,\ldots,x_n\rrbracket$.
Our sequence of ideals determines an ideal
$[\fra_m] \subset \ltightstar{(k[x_1,\dots,x_n])}$ whose image in
$\ltightstar{k}\llbracket x_1,\ldots,x_n\rrbracket$
we denote by $\~\fra$. Our key result is that $\lct(\~\fra) = c$.
After possibly replacing $\ltightstar{k}$ by a larger field $K$, we obtain an
ideal in a polynomial ring over $K$ whose log canonical
threshold is $c$. Since $\cT_n(k)$ is independent of $k$, we get the
conclusion of Theorem~\ref{thm_main1}.
If the sequence $\{c_m\}_m$ is
strictly decreasing, then we conclude that the limit is actually
a log canonical threshold in a smaller dimension via a more careful
analysis of the singularities of the ideal $\widetilde{\fra}$.
We mention that
non-standard methods were also employed in \cite{BMS} to study the
sets of $F$-pure thresholds of hypersurfaces in positive
characteristic (though in that case one could only obtain the
analogue of Theorem~\ref{thm_main1} above).

As it should be apparent from the above sketch of the proof, we need
to work with log canonical thresholds of ideals in formal power
series rings. The familiar framework for studying such invariants is
that of schemes of finite type over a field. However, since
resolutions of singularities are available for arbitrary excellent
schemes (see \cite{Temkin}), it is not hard to extend the theory of
log canonical thresholds and multiplier ideals to such a general
setting. We explain this extension in detail in the next section.

In \S 3 we discuss some elementary properties of the sets $\cT_n$,
in particular the independence of the base field. The proofs of the
main results are contained in \S 4. In \S 5 we make
some comments on Conjecture~\ref{conj_shokurov}, proving
Corollary~\ref{cor_thm_main1} and Proposition~\ref{prop:intro}. The
proposition follows from a basic property of cylinders in the ring
of formal power series. This interpretation of the ACC Conjecture
illustrates once more that formal power series provide the
natural setting when considering sequences of log canonical
thresholds.

After the first version of this article was made public,
J\'anos Koll\'ar gave a new proof of the above theorems
using an infinite sequence of approximations and field extensions
in place of non-standard methods (at the core the two proofs are the same).
In fact, making use of results from \cite{BCHM},
he obtains a stronger version of Theorem~\ref{thm_main2},
showing that all accumulation points of the set of
log canonical thresholds
in dimension $n$ (not just the limits of decreasing sequences)
are log canonical thresholds in dimension $n-1$.

\subsection{Acknowledgments}
The second author is indebted to Caucher Birkar who introduced him
to non-standard constructions. We are grateful to Jarek
W{\l}odarczyk for various discussions on resolution of
singularities, and to Matthias Aschenbrenner and Pierre Milman for
pointing out some key references. We would also like to acknowledge
the comments and suggestions we received from an
anonymous referee.
Our project began at the AIM
workshop on ``Numerical invariants of singularities and
higher-dimensional algebraic varieties''. We would like to thank AIM
for providing a stimulating environment.

\section{Multiplier ideals on excellent schemes}

Our goal in this section is to develop the theory of multiplier
ideals and log canonical thresholds for ideals on a regular
excellent scheme of characteristic zero. We will apply this theory
when the ambient space is either a smooth scheme of finite type over
a field or the spectrum of a formal power series ring over a field.
All our schemes have characteristic zero, that is, they are schemes
over ${\rm Spec}(\QQ)$.

Recall that a Noetherian ring $A$ is \emph{excellent} if the
following hold:
\begin{enumerate}
\item[1)] For every prime ideal $\frp$ in $A$, the completion
morphism $A_{\frp}\to \widehat{A_{\frp}}$  has geometrically regular
fibers.
\item[2)] For every $A$-algebra of finite type $B$, the regular
locus of ${\rm Spec}(B)$ is open.
\item[3)] $A$ is universally catenary.
\end{enumerate}
For the basics on excellent rings we refer to \cite{Matsumura}.
 It is known that every algebra of finite type
over an excellent ring is excellent, and that all complete
Noetherian local rings are excellent. A Noetherian scheme $X$ is
\emph{excellent} if it admits an open cover by spectra of excellent
rings.

The key ingredient in building the theory of multiplier ideals is
the existence of log resolutions of singularities. It was shown in
\cite{Temkin} that Hironaka's Theorem giving existence of
resolutions for integral schemes of finite type over a field
implies the following general statement (in fact, the result in
\emph{loc. cit.} holds for quasi-excellent schemes, but we do not
need this generality).

\begin{theorem}[\cite{Temkin}]\label{thm1}
Let $X$ be an integral, excellent scheme of characteristic zero, and
let $Y\hookrightarrow X$ be a proper closed subscheme. There is a
proper, birational morphism $f\colon X'\to X$ with the following
properties:
\begin{enumerate}
\item[i)] $X'$ is a regular scheme.
\item[ii)] The inverse image $f^{-1}(Y)$ is a divisor with simple
normal crossings.
\end{enumerate}
Moreover, if $U$ is an open subset of $X$ such that $U$ is regular
and $Y\vert_U$ is a divisor with simple normal crossings, then we can
take $f$ to be an isomorphism over $U$.
\end{theorem}

From now on we assume that \emph{all} our schemes are excellent,
of characteristic zero. We want to define the relative canonical divisor $K_{Z/X}$
for a proper birational morphism $f\colon Z\to X$ between two
regular schemes. We consider the ideal $J$ on $Z$ defined as the Fitting ideal
${\rm Fitt}^0(\Omega_{Z/X})$. We first show that this ideal can be computed locally 
in a similar way with the case of schemes of finite type over a field. 
 
Given $z\in Z$ and $x=f(z)$, we consider the injective homomorphism
of regular local rings $\phi\colon \cO_{X,x}\to \cO_{Z,z}$. Since 
this morphism induces an isomorphism of rings of fractions, and since
our rings are in particular universally catenary, it follows from the dimension formula
(see Theorem~15.6 in \cite{Matsumura}) that 
$$\dim(\cO_{X,x})=\dim(\cO_{Z,z})+{\rm trdeg}(k(z)/k(x)).$$
If we choose local parameters $t_1,\ldots,t_r$ for $\cO_{X,x}$,
and $u_1,\ldots,u_s$ for $\cO_{Z,z}$, then the ring homomorphism induced by
$\phi$ at  the level of completions can be described as
$$\widehat{\phi}\colon k(x)\llbracket t_1,\ldots,t_r\rrbracket
\to k(z)\llbracket u_1,\ldots,u_s\rrbracket.$$ Note that
$k(z)$ is a separable extension of $k(x)$, with ${\rm trdeg}(k(z)/k(x))=r-s$.
A basis of ${\rm Der}_{k(x)}(k(z))$ induces derivations $D_1,\ldots,D_{r-s}$
of $k(z)\llbracket u_1,\ldots,u_s\rrbracket$ over $k(x)$, by acting on coefficients.
If we put $\phi_i:=\widehat{\phi}(t_i)$ and $D_{r-s+j}=\partial/\partial u_j$ for $1\leq j\leq s$,
then we see that $J\cdot\widehat{\cO_{Z,z}}$ 
is generated by ${\rm det}(D_i(\phi_j))_{1\leq i,j\leq r}$. 
In particular, $J$ is a locally principal ideal around $z\in Z$. We denote by
$K_{Z/X}$ the effective divisor corresponding to $J$. It defines the locus where the morphism
$f$ is not \'{e}tale. Its complement is the inverse image of an open subset $U$ of $X$,
such that $f^{-1}(U)\to U$ is an isomorphism. 

Suppose now that $X$ is a regular connected 
scheme. If $Y$ is a closed, proper subscheme of $X$, a \emph{log
resolution} $f\colon X'\to X$ of $(X,Y)$ is a morphism as in 
Theorem~\ref{thm1}, such that in addition the union of $f^{-1}(Y)$ with
the exceptional locus of $f$
has simple normal crossings. (Note that this
exceptional locus is automatically a divisor, being equal to the support
of $K_{X'/X}$.)

We see that every $f\colon X'\to X$ as in Theorem~\ref{thm1} is
dominated by some $g\colon X''\to X'$ such that $f\circ g$ is a log
resolution. Indeed, it is enough to apply the theorem for $(X',D)$,
where $D=f^{-1}(Z)+K_{X'/X}$ to get $g$ that is an isomorphism over
$X'\smallsetminus {\rm Supp}(D)$ (in this case, the exceptional locus
of $f\circ g$ is contained in $g^{-1}(D)$).

Furthermore, the above assertion implies that every two log
resolutions of $(X,Y)$ are dominated by a third one.

\bigskip

Let $X$ be a regular, connected excellent scheme of characteristic
zero. Suppose that $Y$ is a proper closed subscheme of $X$, and let
$\fra$ be the ideal sheaf defining $Y$. If $\lambda\in\RR_{\ge 0}$,
then the \emph{multiplier ideal sheaf} $\cJ(\fra^{\lambda})$ is
defined as follows.

Consider a log resolution $f\colon X'\to X$ of the pair $(X,Y)$. If
we denote $f^{-1}(\fra)=\cO_{X'}(-F)$, then
$$
\cJ(\fra^{\lambda}):=f_*\cO_{X'}(K_{X'/X}-\lfloor \lambda F\rfloor).
$$
Since $K_{X'/X}$ is an effective exceptional divisor, we have
$f_*\cO_{X'}(K_{X'/X})=\cO_X$, and therefore $\cJ(\fra^{\lambda})$
is a (coherent) sheaf of ideals in $\cO_X$.

This is, of course, consistent with the usual definition of
multiplier ideals when $X$ is of finite type over a field. Our
reference for the basic results on multiplier ideals is
\cite{positivity}. Most of the results therein extend almost
verbatim to our setting. We state the results for completeness, but
we give detailed proofs only when there is a significant difference
from the familiar case.

\begin{proposition}\label{prop1}
The definition of $\cJ(\fra^{\lambda})$ does not depend on the
choice of a log resolution.
\end{proposition}

\begin{proof}
The argument follows as in \cite{positivity}, Thm.~9.2.18, using the
fact that every two log resolutions can be dominated by a third one, 
and the following consequence of our local computation of $K_{X'/X}$.
If $E$ is a smooth prime divisor on $X'$, and $D_1,\ldots,D_r$ are simple normal crossings
divisors such that $E$ appears with coefficient $a_i\geq 1$ in $f^*(D_i)$, then
$E$ appears with coefficient $\geq a_1+\cdots+a_r-1$ in $K_{X'/X}$.
\end{proof}

Note that if $\fra=\cO_X$, then $\cJ(\fra^{\lambda})=\cO_X$ for
every $\lambda\geq 0$. We make the convention that if $\fra=0$, then
$\cJ(\fra^{\lambda})=0$ for every $\lambda\geq 0$. The following
easy properties follow immediately from the definition.

\begin{proposition}\label{prop2}
Let $X$ be a scheme as above,
 and let $\fra$ be a (coherent) sheaf of ideals on $X$.
\begin{enumerate}
\item[1)] If $\lambda<\mu$, then
$\cJ(\fra^{\mu})\subseteq\cJ(\fra^{\lambda})$.
\item[2)] For every $\lambda\in\RR_{\ge 0}$, there is $\epsilon>0$ such
that $\cJ(\fra^{\mu})=\cJ(\fra^{\lambda})$ for every $\mu\in
[\lambda,\lambda+\epsilon)$.
\item[3)] If $\frb\subseteq\fra$, then
$\cJ(\frb^{\lambda})\subseteq\cJ(\fra^{\lambda})$.
\item[4)] If $\fra\neq 0$, then  $\cJ(\fra^0)=\cO_X$.
\end{enumerate}
\end{proposition}

Suppose now that $\fra\neq 0$ and $\fra\neq\cO_X$. A positive number
$\lambda$ is called a \emph{jumping number} of $\fra$ if
$\cJ(\fra^{\lambda})\neq\cJ(\fra^{\lambda'})$ for every
$\lambda'<\lambda$. It is clear that $\fra$ has jumping numbers.

Consider a log resolution $f\colon X'\to X$, and the divisors $F$
and $K_{X'/X}$, where $f^{-1}(\fra)=\cO(-F)$. Write
\begin{equation}\label{eq_divisors}
F=\sum_{i=1}^ra_iE_i,\,\,K_{X'/X}=\sum_{i=1}^r\kappa_iE_i.
\end{equation}
 It is clear that if $\lambda$ is a jumping number of $\fra$, then there is
$i$ such that $a_i\lambda$ is an integer. In particular, all jumping
numbers are rational numbers. Moreover, the jumping numbers of a
given ideal $\fra$ form a discrete subset of $\RR_+$.

The smallest jumping number of $\fra$ is called the \emph{log
canonical threshold} of $\fra$, and it is denoted by $\lct(\fra)$.
It is the smallest $\lambda$ such that
$\cJ(\fra^{\lambda})\neq\cO_X$. With the notation in
(\ref{eq_divisors}), we have
$$
\lct(\fra)=\min_i\frac{\kappa_i+1}{a_i}.
$$
It follows from Proposition~\ref{prop2} 3) that if
$\frb\subseteq\fra$, then $\lct(\frb)\leq\lct(\fra)$. It is natural
to make the convention that $\lct(0)=0$ and $\lct(\cO_X)=\infty$.

\begin{proposition}\label{prop3}
If $\overline{\fra}$ denotes the integral closure of the ideal
$\fra$, then $\cJ(\fra^{\lambda})=\cJ(\overline{\fra}^{\lambda})$ for every $\lambda$.
In particular, we have $\lct(\fra)=\lct(\overline{\fra})$.
\end{proposition}

\begin{proof}
This is an immediate consequence of the fact that if $f\colon X'\to
X$ is a log resolution of both $\fra$ and $\overline{\fra}$, then
$f^{-1}(\fra)=f^{-1}(\overline{\fra})$. See also \cite{positivity},
\S 9.6.A.
\end{proof}

It is convenient to define also a local version of the log canonical
threshold: if $\xi\in X$ is a (not necessarily closed) point in
$V(\fra)$, then we define
$$
\lct_{\xi}(\fra):=\min\{\lct(\fra\vert_U)\mid U\subseteq X\,{\rm
open},\,\xi\in U\}.
$$
Using the notation in (\ref{eq_divisors}), we have
$$
\lct_{\xi}(\fra)=\min\{(\kappa_i+1)/a_i\mid \xi\in f(E_i)\}.
$$
Note that if $X={\rm Spec}(R)$, where $R$ is a local ring, then $X$
has only one closed point $x_0$, and $\lct_{x_0}(\fra)=\lct(\fra)$.

The \emph{log canonical threshold locus} $\LCT(\fra)$ of $\fra$ is
the closed subset of $X$ where $\cJ(\fra^{\lct(\fra)})$ vanishes. More
generally, for every (not necessarily closed) point $\xi$ on $X$, we
denote by $\LCT_\xi(\fra)$ the support of the closed subscheme of
$X$ defined by $\cJ(\fra^{\lct_\xi(\fra)})$. With the notation in
(\ref{eq_divisors}), we have
\begin{equation}\label{formula_LCT}
\LCT(\fra)=\bigcup_{\frac{\kappa_i+1}{a_i}=\lct(\fra)}f(E_i) \and
\LCT_\xi(\fra)= \bigcup_{\frac{\kappa_i+1}{a_i}\le
\lct_\xi(\fra)}f(E_i).
\end{equation}
Note that $\LCT_{\xi}(\fra)\cap U=\LCT(\fra\vert_U)$ for a suitable
open neighborhood $U$ of $\xi$.

Suppose that $X$ is a scheme as above, and let $\fra$ be a nonzero
ideal sheaf on $X$. We consider a point $\xi\in X$, and let $\frp$
denote the ideal defining the closed set $\overline{\{\xi\}}$ (with
the reduced scheme structure). Suppose that $\xi$ lies in the
support of the subscheme defined by $\fra$ (that is, $\frp \supseteq
\fra$). Since $\fra \subseteq \fra + \frp^d$, we clearly have
$$
\lct_{\xi}(\fra+\frp^d) \ge \lct_{\xi}(\fra) \fall d \ge 1.
$$
In fact, the following key property holds.

\begin{proposition}\label{limit1}
With the above notation, we have
$$
\lim_{d\to\infty}\lct_{\xi}(\fra+\frp^d) = \lct_{\xi}(\fra).
$$
Moreover, if $\overline{\{\xi\}}$ is an irreducible component of
$\LCT_\xi(\fra)$, then $\lct_{\xi}(\fra+\frp^d) = \lct_{\xi}(\fra)$
for all $d \gg 0$.
\end{proposition}

We first prove the following lemma. Recall that a \emph{divisor over
$X$} is a prime divisor $E$ on some $X'$, where we have a proper
birational morphism $f\colon X'\to X$, with $X'$ regular. The
\emph{center} of $E$ on $X$ is $c_X(E):=f(E)$. Such a divisor
defines a valuation ${\rm ord}_E$ of the function field of $X$. In
particular, we may define ${\rm ord}_E(K_{X'/X})$ and ${\rm
ord}_E(\fra)$ in the natural way. Note that if we use the notation
in (\ref{eq_divisors}), then ${\rm ord}_{E_i}(K_{X'/X})=\kappa_i$
and ${\rm ord}_{E_i}(\fra)=a_i$.

\begin{lemma}
With the notation in Proposition~\ref{limit1}, we have
\begin{equation}\label{eq_inf}
\lct_\xi(\fra)=\inf_{c_X(E)=
\overline{\{\xi\}}}\frac{\ord_E(K_{X'/X})+1}{\ord_E(\fra)},
\end{equation}
where the infimum is taken over all divisors $E$ over $X$ whose
center on $X$ is equal to $\overline{\{\xi\}}$. Moreover, if
$\overline{\{\xi\}}$ is an irreducible component of
$\LCT_\xi(\fra)$, then there is a divisor $E$ with center
$\overline{\{\xi\}}$ on $X$, such that
$$
\lct_\xi(\fra)=\frac{\ord_E(K_{X'/X})+1}{\ord_E(\fra)}.
$$
\end{lemma}

\begin{proof}
Note first that the infimum in (\ref{eq_inf}) is
$\geq\lct_{\xi}(\fra)$ by definition. In order to prove the other
inequality, the argument follows as in \cite[(17.1.1.3)]{Ketal}.
Let us consider a log resolution $f\colon X'\to X$ for
$\fra\cdot\frp$ (recall that $\frp$ is the ideal defining
$\overline{\{\xi\}}$). With the notation in (\ref{eq_divisors}), fix
a prime divisor $F=E_i$ on $X'$ such that $\xi\in c_X(F)$ and
$\lct_{\xi}(\fra)=\frac{\kappa_i+1}{a_i}$.

If $c_X(F)=\overline{\{\xi\}}$, then we are done. Otherwise,
by construction, $f^{-1}(\overline{\{\xi\}})$ is a union of divisors
having simple normal crossings with $F$. Hence we may choose such a
divisor $F_0$, with $F\cap F_0$ containing a point $\xi'\in
f^{-1}(\xi)$. This implies that $f(F\cap F_0)=\overline{\{\xi\}}$.
We define recursively $F_m$ for $m\geq 1$, as follows. Let $X'_1$ be
the blow-up of $X'$ along $F\cap F_0$, and $F_1$ be a component of
the exceptional divisor whose image in $X'$ contains $\xi'$. In
general, we denote by $X'_m$ the blow-up of $X'_{m-1}$ along the
intersection of $F_{m-1}$ with the proper transform of $F$, and by
$F_m$ the component of the exceptional divisor whose image in $X'$
contains $\xi'$. It is clear that $c_X(F_m)=\overline{\{\xi\}}$ for
every $m$. On the other hand, a standard computation gives
$${\rm ord}_{F_m}(\fra)=m\cdot {\rm ord}_F(\fra)+{\rm
ord}_{F_0}(\fra),$$
$$1+{\rm ord}_{F_m}(K_{X'_m/X})=m\left(1+{\rm ord}_{F}(K_{X'/X})\right)+\left(1+{\rm
ord}_{F_0}(K_{X'/X})\right).$$ Therefore we get (\ref{eq_inf}),
since
$$\lim_{m\to\infty}\frac{1+{\rm ord}_{F_m}(K_{X'_m/X})}{{\rm
ord}_m(\fra)}=\lct_{\xi}(\fra).$$

The second assertion, when $\overline{\{\xi\}}$ is an irreducible
component of $\LCT_{\xi}(\fra)$, follows from (\ref{formula_LCT}).
Indeed, note that if $\xi\in f(E_i)$, then by definition we have
\begin{equation}\label{ineq10}
\lct_\xi(\fra)\leq\frac{\ord_{E_i}(K_{X'/X})+1}{\ord_{E_i}(\fra)},
\end{equation}
and the inequality is strict if $\overline{\{\xi\}}$ is strictly contained
in $f(E_i)$.

Therefore every irreducible component of $\LCT_{\xi}(\fra)$
containing $\xi$ is of the form $f(E_i)$ for some $E_i$ that
achieves equality in (\ref{ineq10}).
\end{proof}

\begin{proof}[Proof of Proposition~\ref{limit1}]
By the lemma, we see that for every $\epsilon >0$ there is a divisor
$E$ on some $X'$ over $X$ having center $\overline{\{\xi\}}$ and such that
$$
\frac{\ord_E(K_{X'/X})+1}{\ord_E(\fra)} \le \lct_\xi(\fra) +
\epsilon.
$$
Since $\frp$ is the ideal defining $\overline{\{\xi\}}$, we have
$\ord_E(\frp) \ge 1$. It follows that if $d>{\rm ord}_E(\fra)$, then
 $\ord_E(\fra + \frp^d) =
\ord_E(\fra)$. Hence for all such $d$ we have
$$
\lct_\xi(\fra + \frp^d) \le \frac{\ord_E(K_{X'/X})+1}{\ord_E(\fra +
\frp^d)} =
\frac{\ord_E(K_{X'/X})+1}{\ord_E(\fra)}\leq\lct_{\xi}(\fra)+\epsilon.
$$
Therefore $\limsup_{d \to \infty}\lct_{\xi}(\fra+\frp^d) \le
\lct_{\xi}(\fra)$. Since, as we have already observed,
$\lct_{\xi}(\fra+\frp^d) \ge \lct_{\xi}(\fra)$ for every $d \ge 1$,
we conclude that the limit exists and is equal to
$\lct_{\xi}(\fra)$.

Regarding the second assertion, if $\overline{\{\xi\}}$ is an irreducible
component of $\LCT_\xi(\fra)$, then by the second part of the lemma
we can find an $E$ with center $\overline{\{\xi\}}$ such that
$$
\frac{\ord_E(K_{X'/X})+1}{\ord_E(\fra)} = \lct_\xi(\fra),
$$
and thus we obtain that $\lct_\xi(\fra + \frp^d) = \lct_\xi(\fra)$
for all $d>{\rm ord}_E(\fra)$.
\end{proof}

Our next result allows us to reduce the computation of multiplier
ideals and log canonical thresholds to the case when $X$ is the
spectrum of a complete regular local ring of characteristic zero.

\begin{proposition}\label{completion}
Let $X$ be a regular, connected excellent scheme of characteristic
zero. Suppose that $\xi\in X$ is a (not necessarily closed) point,
and let $\widehat{\cO_{X,\xi}}$ be the completion of the local ring
at $\xi$. If $h\colon {\rm Spec}(\widehat{\cO_{X,\xi}})\to X$ is the
canonical morphism, then for every nonzero sheaf of ideals $\fra$ on
$X$ such that $\xi$ lies in the support of the subscheme defined by
$\fra$, and for every $\lambda\in\RR_+$ we have
$$
\cJ(\fra^{\lambda})\.\widehat{\cO_{X,\xi}}=\cJ\big((\fra\.
\widehat{\cO_{X,\xi}})^{\lambda}\big).
$$
In particular, we have $\lct_{\xi}(\fra)=\lct(\fra\.
\widehat{\cO_{X,\xi}})$.
\end{proposition}

We will make use, in particular, of the following special case.

\begin{corollary}\label{reduction1}
Let $X$ and $\fra$ be as in the theorem. If
$\xi$ is the generic point of an irreducible component of
$\LCT(\fra)$, then $\lct(\fra)=\lct(\fra\cdot
\widehat{\cO_{X,\xi}})$, and $\LCT(\fra\cdot
\widehat{\cO_{X,\xi}})$ consists only of the closed point.
\end{corollary}

\begin{proof}[Proof of Proposition~\ref{completion}]
Let $W={\rm Spec}(\widehat{\cO_{X,\xi}})$, and denote by $Y$ and $Z$
the subschemes of $X$ and $W$, respectively defined by $\fra$ and
$\fra\. \widehat{\cO_{X,\xi}}$. Let $f\colon X'\to X$ be a log
resolution of $(X,Y)$. If we consider the Cartesian diagram
\begin{equation}
\begin{CD}\label{diagram}
W' @>{h'}>> X' \\
@VV{g}V @VV{f}V \\
W @>{h}>>X
\end{CD}
\end{equation}
then $g$ is a log resolution of $(W,Z)$. Indeed, it is clear that
$g$ is proper, and if $f$ is an isomorphism over the nonempty open
subset $U$, then $g$ is an isomorphism over
$h^{-1}(U)$, which is nonempty. Furthermore, $W'$ is a regular scheme. In
order to see this, it is enough to show that all fibers of $h'$ are
regular (note that $h$ is flat, and therefore $h'$ is flat, too, and
we may apply Thm. 23.7 in \cite{Matsumura}). This in turn follows
from the fact that the morphism
$\cO_{X,\xi}\to\widehat{\cO_{X,\xi}}$ has geometrically regular
fibers (see the first property in the definition of excellent
rings).

Using the same argument, we see that if $x_1,\ldots,x_n$ are
algebraic coordinates in an open subset $V$ of $X'$, then their
pull-backs to ${h'}^{-1}(V)$ give an algebraic system of coordinates
in this open subset of $W'$. Note also that
$K_{W'/W}={h'}^*(K_{X/X})$. Putting these together, we see that
$K_{W'/W}+g^{-1}(Z)$ is a divisor with simple normal crossings on
$W'$, and therefore $g$ is a log resolution of $(W,Z)$. Moreover, if
$F=f^{-1}(Y)$, then ${h'}^*(F)=g^{-1}(Z)$, and flat base change
gives
$$\cJ\big((\fra\. \widehat{\cO_{X,\xi}})^{\lambda}\big)
=g_*\cO_{W'}\left({h'}^*(K_{X'/X}-\lfloor \lambda
F\rfloor)\right)=h^*f_*\cO_{X'}(K_{X'/X}-\lfloor \lambda
F\rfloor)=\cJ(\fra^{\lambda})\.\widehat{\cO_{X,\xi}}.$$
The assertion
about log canonical thresholds is an immediate consequence, since
$$\lct_{\xi}(\fra)=\sup\{\lambda\mid \cJ(\fra^\lambda)\cdot\cO_{X,\xi}=
\cO_{X,\xi}\}=\sup\{\lambda\mid \cJ(\fra^\lambda)\cdot\widehat{\cO_{X,\xi}}=
\widehat{\cO_{X,\xi}}\}=\lct(\fra\. \widehat{\cO_{X,\xi}}).$$
\end{proof}

\begin{proposition}\label{field_extension}
Let $K/k$ be a field extension. If $\fra$ is a nonzero ideal in
$k[x_1,\ldots,x_n]$ whose cosupport contains the point $0$
corresponding to $(x_1,\ldots,x_n)$, then
\begin{equation}\label{formula_1}
\lct_0(\fra)=\lct_0(\fra\cdot K[x_1,\ldots,x_n]).
\end{equation}
 Similarly, if
$\frb$ is a nonzero proper ideal in $k\[x_1,\ldots,x_n\]$, then
\begin{equation}\label{formula_2}
\lct(\frb)=\lct(\frb\cdot K\[x_1,\ldots,x_n\]).
\end{equation}
\end{proposition}

\begin{proof}
Let $X={\mathbf A}^n_k$ and $W={\mathbf A}_K^n$, and let $h\colon
W\to X$ denote the induced flat map. It follows as in the proof of
Proposition~\ref{completion} that if $f$ is a log resolution of
$(X,\fra)$, then we get a Cartesian diagram (\ref{diagram}) such
that $g$ is a log resolution of $(W,\fra\.\cO_W)$ (the key fact is
that since we are in characteristic zero, the field extension $K/k$ is
separable, hence geometrically regular). As a consequence,
we deduce that
$$\cJ\left((\fra\cdot\cO_W)^{\lambda}\right)=\cJ(\fra^{\lambda})\cdot\cO_W$$
for every $\lambda\in\RR_+$, and since the homomorphism
$\cO_{X,0}\to\cO_{W,0}$ is faithfully flat, we get
(\ref{formula_1}).

In the formal power series case, denote by $\frm$ and $\frm'$ the
maximal ideals in $k\[x_1,\ldots,x_n\]$ and $K\[x_1,\ldots,x_n\]$,
respectively. For every $d$, the ideal $\frb+\frm^d$ is generated by
polynomials of degree $\leq d$, hence it is equal to $\frb_d\cdot
k\[x_1,\ldots,x_n\]$, where $\frb_d=(\frb+\frm^d)\cap
k[x_1,\ldots,x_n]$. If we put $\frb_d'= \frb_d\. K[x_1,\ldots,x_n]$,
then it follows from (\ref{formula_1}) that
$\lct_0(\frb_d)=\lct_0(\frb'_d)$ for every $d$. On the other hand,
Proposition~\ref{completion} gives
$$\lct(\frb_d\. k\[x_1,\ldots,x_n\])=\lct_0(\frb_d)\quad\text{and}\quad
\lct(\frb'_d\. K\[x_1,\ldots,x_n\]=\lct_0(\frb'_d).$$ Using
Proposition~\ref{limit1}, we deduce that
$$\lct(\frb)=\lim_{d\to\infty}\lct_0(\frb_d) \quad\text{and}\quad
\lct(\frb\. K\[x_1,\ldots,x_n\])=\lim_{d\to\infty}\lct_0(\frb'_d),$$
and since $\lct_0(\frb_d) = \lct_0(\frb'_d)$, we get (\ref{formula_2}).
\end{proof}

We note that using Propositions~\ref{limit1} and \ref{completion}
one can extend various results about log canonical
thresholds from the familiar case of varieties over an algebraically
closed field to our more general setting. We state one such result
which gives a sharpening of Proposition~\ref{limit1}.

\begin{corollary}\label{limit2}
Let $X$ be a connected, regular excellent scheme, and let $\fra$ be
a nonzero ideal sheaf on $X$. We consider a point $\xi$ in the cosupport of
$\fra$, and let $\frp$ be the ideal sheaf defining
$\overline{\{\xi\}}$. If $e=\dim(\cO_{X,\xi})$, and if $\frb$ is an
ideal sheaf on $X$ such that $\frb\subseteq\fra+\frp^{(d)}$ $\rm
($where $\frp^{(d)}$ denotes the $d^{\rm th}$ symbolic power of
$\frp$$\rm )$, then
$$
\lct_{\xi}(\frb)-\lct_{\xi}(\fra)\leq \frac{e}{d}.
$$
\end{corollary}

\begin{proof}
 After replacing $\fra$ and $\frb$ by their images in
$\widehat{\cO_{X,\xi}}$, and applying Proposition~\ref{completion},
we see that we may assume that $X={\rm Spec}(k\llbracket
x_1,\ldots,x_e\rrbracket)$ and that $\xi=\frm$ is the closed point.
Note that in this case the hypothesis simply says that
$\frb\subseteq\fra+\frm^d$, and it is enough to show that
$$\lct(\fra+\frm^d)-\lct(\fra)\leq\frac{e}{d}.$$
 Furthermore, using
Proposition~\ref{field_extension}, we may replace $k$ by an
algebraic closure, and therefore assume that $k$ is algebraically
closed.

If we put $\fra_q=\fra+\frm^q$, then by Proposition~\ref{limit1} we
have
$$\lct(\fra)=\lim_{q\to\infty}\lct(\fra_q)
\quad\text{and}\quad
\lct(\fra+\frm^d)=\lim_{q\to\infty}
\lct(\fra_q+\frm^d).$$ Therefore it is enough to prove the corollary
when we replace $\fra$ by $\fra_q$. Since each $\fra_q$ is the
extension of an ideal in $k[x_1,\ldots,x_e]$, it follows that it is
enough to prove the corollary when $X={\mathbf A}_k^e$ and $\xi$
is the origin $0$.
In this case, the assertion is well-known: it follows, for example,
from the fact that
$$\lct_0(\fra+\frm^d)\leq\lct_0(\fra)+\lct_0(\frm^d)=\lct_0(\fra)+\frac{e}{d}$$
(see Proposition~4.7 in \cite{Mus}).
\end{proof}

Note that every complete regular local ring of characteristic zero
(or more generally, containing a field) is isomorphic to a formal
power series ring $k\[x_1,\ldots,x_n\]$. It follows from
Corollary~\ref{reduction1} that the log canonical threshold of every
ideal $\fra$ in our general setting is equal to the log canonical
threshold of an ideal in a formal power series ring whose log
canonical threshold locus is supported at the closed point. We now
show that for every such ideal, its log canonical threshold is equal
to the log canonical threshold of a zero-dimensional ideal in a
polynomial ring over an algebraically closed field.

\begin{proposition}\label{reduction_to_zero-dimensional}
Let $X$ be a connected, regular excellent scheme of dimension $N$.
For every proper nonzero ideal sheaf $\fra$ on $X$, there is an
algebraically closed field $k$ of characteristic zero and an ideal
$\frb$ in a polynomial ring $k[x_1,\ldots,x_n]$ such that $\frb$ is
cosupported at the origin $0$, and $\lct(\fra)=\lct_0(\frb)$.
Moreover,
\begin{enumerate}
\item[i)] We may take $n\leq N$.
\item[ii)] If $\dim\,\LCT(\fra)>0$, then we may take $n<N$.
\item[iii)] If $\dim\,\LCT(\fra)=0$ and $X$ is a scheme of finite type over an
algebraically closed field $K$, then we may take $k=K$.
\end{enumerate}
\end{proposition}

\begin{proof}
Let $\xi$ be the generic point of an irreducible component of
$\LCT(\fra)$. We fix an isomorphism $\widehat{\cO_{X,\xi}} \cong
k_0\llbracket x_1,\dots,x_n\rrbracket$, where $k_0$ is the residue
field of $\widehat{\cO_{X,\xi}}$. By Corollary~\ref{reduction1}, the
image of $\fra\cdot \widehat{\cO_{X,\xi}}$  via this isomorphism is
an ideal $\~\fra$ of $k_0\llbracket x_1,\dots,x_n\rrbracket$ with
$\lct(\~\fra) = \lct(\fra)$ and $\LCT(\~\fra) = \{\~\frm\}$ (here
$\~\frm$ denotes the maximal ideal of $k_0\llbracket
x_1,\dots,x_n\rrbracket$). Proposition~\ref{limit1} gives
$\lct(\~\fra + \~\frm^d) = \lct(\~\fra)$ for all $d \gg 0$, and
clearly $\LCT(\~\fra + \~\frm^d) = \{\~\frm\}$ for such $d$. We fix a
sufficiently large $d$, and observe that the ideal $\~\fra +
\~\frm^d$ is generated by polynomials of degree $\le d$, hence
$\~\fra + \~\frm^d = \frc\.k_0\llbracket x_1,\dots,x_n\rrbracket$ for
some ideal $\frc \subseteq k_0[x_1,\dots,x_n]$. We put $\frb=\frc
\cdot k[x_1,\ldots,x_n]$, where $k$ is an algebraic closure of
$k_0$. Since $\lct(\~\fra + \~\frm^d) = \lct_0(\frc)$ by
Proposition~\ref{completion}, and $\lct_0(\frb) =\lct_0(\frc)$ by
Proposition~\ref{field_extension}, we deduce that
$\lct(\fra)=\lct_0(\frb)$.

It is clear from the above construction that $n\leq N$, and the
inequality is strict if $\xi$ is not a closed point. This gives i)
and ii) in the proposition. Moreover, if $X$ is of finite type over
the algebraically closed field $K$, and $\xi$ is a closed point,
then $k_0=K$, which gives iii).
\end{proof}

\section{Sets of log canonical thresholds}

Let $k$ be an algebraically closed field of characteristic zero. For
every integer $n \ge 0$ we consider the following subsets of $\R$:
\begin{align*}
\cT_n(k) &:= \{ \lct(X,Y) \mid \text{$X$ smooth variety over $k$ of
dimension $n$,
$\emptyset \ne Y \subseteq X$} \} \\
\cTpol_n(k) &:= \{ \lct_0(\fra) \mid
\fra \subseteq (x_1,\dots,x_n) k[x_1,\dots,x_n] \} \\
\cTser_n(k) &:= \{ \lct(\~\fra) \mid \~\fra \subseteq
(x_1,\dots,x_n) k\llbracket x_1,\dots,x_n\rrbracket \}.
\end{align*}
Note that all these sets are contained in $\Q$. We also consider the
set $\cTiso_n(k) \subseteq \cT_n(k)$ of log canonical thresholds of
pairs $(X,Y)$, with $X$ smooth and $n$-dimensional, $Y$ nonempty,
and such that the log canonical threshold locus $\LCT(X,Y)$ is
zero-dimensional (by convention, we put $\cTiso_0(k)=\{0\}$).

It is clear that we have $\cT_{n-1}(k)\subseteq\cT_n(k)$ for every
$n\geq 1$ (and similar inclusions for the other sets). Indeed, this
follows from the fact that
$\lct(X,Y)=\lct(X\times\A^1,Y\times\A^1)$.

\bigskip

Before discussing some basic properties of the sets we have just
introduced, we make some general remarks about log canonical
thresholds of polynomials of bounded degree. We start by recalling
an interpretation of the log canonical threshold in terms of jet
schemes from \cite{Mus}. Let us fix an algebraically closed field
$k$ of characteristic zero. Recall that if $X$ is a scheme of finite
type over $k$, then the $m^{\rm th}$ jet scheme of $X$ is a scheme
$X_m$ of finite type over $X$, such that the $k$-points of $X_m$ are
in natural bijection with ${\rm Hom}(\Spec\,k[t]/(t^{m+1}),X)$. If
$P\in X$ is a point, then we denote by $X_{m,P}$ the fiber of $X_m$
over $P$. It is proved in Corollary~3.6 in \emph{loc. cit.} that if $X$
is smooth and $Y\hookrightarrow X$ is a closed subscheme containing
the point $P\in X$, then
\begin{equation}\label{formula_lct}
\lct_P(X,Y)=\dim(X)-\sup_{m\geq 0}\frac{\dim(Y_{m,P})}{m+1}.
\end{equation}
An important remark for our applications is the fact that if we have
a family of subschemes ${\mathcal Y}\hookrightarrow S\times {\mathbb
A}^N$ parametrized by $S$ and defined over a subfield $k_0$ of $k$,
then we have a closed subscheme $({\mathcal Y}/S)_m\hookrightarrow
S\times ({\mathbb A}^N)_m$ defined over $k_0$, whose fiber over a
point $t\in S$ is the $m^{\rm th}$ jet scheme of the fiber of
${\mathcal Y}$ over $t$.

We now turn to describing the set of ideals generated in bounded
degree and having constant log canonical threshold. Suppose that $L$
is an algebraically closed field containing our base field $k$.
Since there are $r: ={{n+d}\choose n}-1$ monomials of positive
degree $\leq d$ in $L[x_1,\ldots,x_n]$, it follows that every ideal
$\frb\subseteq (x_1,\ldots,x_n)L[x_1,\ldots,x_n]$ generated in
degree $\leq d$ can be generated by $r$ linear combinations of these monomials.
Hence we can find
a parameter space $W=\A^{r^2}$ such that for every such $L$, we have
a functorial surjective map
$$
\phi_{L}\colon W(L)\to\{\frb\subseteq
(x_1,\ldots,x_n)L[x_1,\ldots,x_n]\mid \frb\,\text{generated in
degree}\,\leq d\}.
$$
Here we denote by $W(L)$ the set of $L$-valued points of $W$.
There is also a closed subscheme ${\mathcal Y}\hookrightarrow
W\times \A^n$ defined over $\QQ$, such that the fiber of
${\mathcal Y}$ over $u\in W(L)$ is the subscheme defined by the
ideal $\phi_L(u)$.

For every $L$ as above, and for every $\lambda\in\R$, consider the
set
$$
W(L)_{<\lambda}:=\{u\in W(L)\mid\lct_0(\phi_L(u))<\lambda\}.
$$
This is a closed subset of $W(L)$ by the semicontinuity property of
log canonical thresholds (see Corollary~9.5.39 in \cite{positivity},
or Theorem~4.9 in \cite{Mus}). On the other hand, consider for every
$i\in\NN$ the closed subset $S_i$ of $W$ such that for every $L$,
the points in $S_i(L)$ correspond to those $u\in W(L)$ having the
fiber over $0$ of the $i^{\rm th}$ jet scheme of $V(\phi_L(u))$ of
dimension $> (i+1)(n-\lambda)$ (the fact that $S_i$ is closed in $W$
follows using the ${\mathbf G}_m$-action on jet schemes, see
Proposition~2.3 in \cite{Mus}). We deduce from (\ref{formula_lct})
that for every $L$, we have $W(L)_{<\lambda}=\bigcup_{i\geq
0}S_i(L)$. A key point is that each $S_i$ is defined over our base
field $k$ (in fact, over the algebraic closure $\overline{\QQ}$ of
$\QQ$).

Let us now fix an extension $L$ of $k$ that is uncountable (and
algebraically closed). Since $\bigcup_{i\geq 0}S_i(L)$ is closed, it
follows that there is $s\in\NN$ such that $S_i\subseteq \cup_{j=0}^s
S_j$ for every $i$. If we put $W_{<\lambda}:=\cup_{j=0}^s S_j$, then
we see that this is a closed algebraic subset of $W$ defined over
$k$, such that for every extension $K$ of $k$, the points of
$W_{<\lambda}(K)\subseteq W(K)$ correspond to $W(K)_{<\lambda}$.
Note that there are only finitely many distinct $W_{<\lambda}$
as $\lambda$ varies in $\R$: this
follows from the fact that over $k$ there are only finitely many
possible log canonical thresholds corresponding to ideals
parametrized by a scheme of finite type over $k$ (see, for example,
Lemma~4.8 in \cite{Mus}). The above discussion implies that the same
will hold over every field:

\begin{proposition}\label{lem:finitely-many-lct}
For every $d$, the set of log canonical
thresholds of ideals of $k[x_1,\dots,x_n]$
generated by polynomials of degree $\leq d$ is finite
and independent of the ground field $k$.
In particular, the set $\cTpol_n(k)$ is independent of the (algebraically closed)
field $k$.
\end{proposition}

\begin{proof}
After a linear change of coordinates, we see that it is enough to
consider log canonical thresholds at the origin. In this case, it is
enough to run the above argument with $k$ replaced by
$\overline{\QQ}$. We have already mentioned that there are only
finitely many possible log canonical thresholds. Moreover, $\lambda$
really is such a log canonical threshold if and only if $W_{<\lambda}\neq W_{<\lambda'}$
for every $\lambda'>\lambda$. This condition is independent of the ground field, hence 
our assertion.
\end{proof}

\begin{remark}
Of course, the above proposition can be also proved using the description of log canonical thresholds
in terms of log resolutions. However, we decided to give the above argument using jet schemes, since
in the next section we will need to make use of this setting anyway.
\end{remark}

\begin{proposition}\label{prop:all-the-same-subset}
For every $n$ and every $k$ we have $\cT_n(k) = \cTpol_n(k) =
\cTser_n(k)$.
\end{proposition}

\begin{proof}
The inclusion $\cTpol_n(k)\subseteq\cT_n(k)$ is trivial, while
$\cTser_n(k)\subseteq \cTpol_n(k)$ follows from
Proposition~\ref{reduction_to_zero-dimensional} and
Proposition~\ref{lem:finitely-many-lct}. On the other hand, we have
$\cT_n(k)\subseteq\cTser_n(k)$ by Proposition~\ref{completion}: if
$c=\lct(\fra)$, where $\fra$ is an ideal on $X$, and if
$p\in\LCT(\fra)$, then
$c=\lct_p(\fra)=\lct(\fra\cdot\widehat{\cO_{X,p}})$. Since
$\dim(X)=n$, we have $\widehat{\cO_{X,p}}\simeq
k\[x_1,\ldots,x_n\]$, hence our assertion.
\end{proof}

In light of Propositions~\ref{lem:finitely-many-lct} and
\ref{prop:all-the-same-subset}, from now on we simply write $\cT_n$
for either of $\cT_n(k)$, $\cTpol_n(k)$, or $\cTser_n(k)$.

\begin{proposition}\label{union_of_is}
For every $n\geq 1$ we have
\begin{enumerate}
\item[i)] $\cTiso_n(k)$ is independent of $k$ $($hence we simply
write $\cTiso_n$ instead of $\cTiso_n(k)$$)$.
\item[ii)] $\cT_n=\cT_{n-1}\cup\cTiso_n$.
\item[iii)] $\cT_n=\bigcup_{i=0}^n\cTiso_i$.
\end{enumerate}
\end{proposition}

\begin{proof}
The proof of~i) is similar to that of
Proposition~\ref{lem:finitely-many-lct}, so we just describe the
required modifications. If $k\subseteq K$ are algebraically closed
fields, then the inclusion $\cTiso_n(k)\subseteq\cTiso_n(K)$ is
clear: it is enough to use Proposition~\ref{field_extension}.

 For the reverse inclusion, note that if $c\in
\cTiso_n(K)$, then by
Proposition~\ref{reduction_to_zero-dimensional} we can write
$c=\lct_0(\fra)$ for an ideal $\fra\subseteq
(x_1,\ldots,x_n)\.K[x_1,\ldots,x_n]$ cosupported at $\{0\}$. Fix
$d\gg 0$ such that $x_i^d\in\fra$ for every $i$, and such that
$\fra$ is generated in degree $\leq d$. We mimic the construction
preceding Proposition~\ref{lem:finitely-many-lct}, replacing
$W=\A^{r^2}$ by the parameter space $W'=\A^{(r-n)^2}$ defined so
that $W'(K)$ gives all ideals generated in degree $\leq d$ and
containing all the $x_i^d$. We see that there is a locally closed
subset $A$ of $W'$ defined over $\overline{\QQ}$, such that $A(K)$
corresponds to those ideals having log canonical threshold $c$.
Since $A\neq\emptyset$, we have $A(k)\neq\emptyset$, hence we find a
zero-dimensional ideal $\frb\subseteq
(x_1,\ldots,x_n)\.k[x_1,\ldots,x_n]$ cosupported at the origin and
such that $c=\lct(\frb)$. We clearly have $\LCT(\frb)=\{0\}$, hence
$c\in\cTiso_n(k)$. This completes the proof of i).

For~ii), fix an algebraically closed field $k$. It follows from
definition that $\cT_{n-1}\cup\cTiso_n\subseteq \cT_n$. Suppose now
that $c\in\cT_n\smallsetminus\cTiso_n$. We have
$c=\lct(X,Y)$, where $X$ is an $n$-dimensional smooth variety over
$k$ and $\dim\,\LCT(X,Y)>0$. It follows from
Proposition~\ref{reduction_to_zero-dimensional} that
$c\in\cT_{n-1}$, hence ii). The assertion in~iii) is an immediate
consequence.
\end{proof}

\begin{remark}
The set $\cTiso_n$ is dense in $\cT_n$. In fact, every $c\in\cT_n$
is the limit of a decreasing (possibly constant) sequence in
$\cTiso_n$. Indeed, if $c=\lct_x(\fra)$, for an ideal
$\fra\subseteq\frm_x$ on the $n$-dimensional smooth variety $X$
(here $\frm_x$ is the ideal of the point $x$), then
Proposition~\ref{limit1} shows that
$\{\lct_x(\fra+\frm_x^{\ell})\}_{\ell}$ is a decreasing sequence in
$\cTiso_n$ converging to $c$.
\end{remark}

We now define similar sets by considering only pairs $(X,Y)$, with
$Y$ locally principal. More precisely, we denote by $\cHT_n(k)$ the
set of all $\lct(X,Y)$, where $X$ is a smooth variety over $k$ of
dimension $n$ and $Y$ is a nonempty closed subscheme locally
defined by one equation. We similarly put
\begin{align*}
 {\cHTpol_n(k)} &:= \{ \lct_0(f) \mid
f \in (x_1,\dots,x_n) k[x_1,\dots,x_n] \} \\
{\cHTser_n(k)} &:= \{ \lct(\~f) \mid \~f \in (x_1,\dots,x_n)
k\llbracket x_1,\dots,x_n\rrbracket \}.
\end{align*}
We also define $\cHTiso_n(k) \subseteq \cHT_n(k)$ by
requiring that $\LCT(X,Y)$ is zero-dimensional. For future
reference, we record in the following proposition some easy
properties of the sets $\cT_n$ and $\cHT_n$.

\begin{proposition}\label{easy_properties}
Given $n\geq 1$, we have
\begin{enumerate}
\item[i)] $\cT_n\subseteq [0,n]\cap\QQ$.
\item[ii)] $\cHT_n(k)=\cT_n(k)\cap [0,1]$
and $\cHTpol_n(k)=\cTpol_n(k)\cap [0,1]$.
\item[iii)] For every positive integer
$m$, we have $\frac{1}{m}\.\cT_n\subseteq \cT_n$ and
$\frac{1}{m}\.\cHT_n(k)\subseteq\cHT_n(k)$.
\item[iv)] $\cT_n\subseteq n\cdot \cHT_n(k)$.
\end{enumerate}
\end{proposition}

\begin{proof}
All assertions are well-known. For~i) one uses the fact that if
$\fra$ is an ideal sheaf on a smooth variety $X$ vanishing at the point
$P$, and if $E$ is the exceptional divisor on the blowing-up $X'$ of
$X$ at $P$, then
$$\lct_P(\fra)\leq\frac{\ord_E(K_{X'/X})+1}{\ord_E(\fra)}
=\frac{n}{\ord_E(\fra)}\leq n.$$

For the assertion in~ii) we use the fact
that if $\fra\subset\cO_X$ is a proper ideal, with $X$ affine, and
if $f\in\fra$ is a general linear combination of a set of
generators, then
 $\cJ(\fra^{\lambda})=\cJ(f^{\lambda})$ for every $\lambda<1$
 (see Prop. 9.2.28 in \cite{positivity}).
 In particular, if $\lct(\fra)\leq 1$, then $\lct(\fra)=\lct(f)$.

Note that~iii) follows
from the fact that $\lct(\fra^m)=\tfrac 1m \lct(\fra)$, and~iv)
is a consequence of~i)--iii).
\end{proof}

\begin{corollary}\label{cor:independence-of-k}
For every $n$, the set $\cHT_n(k)$ is independent of the
algebraically closed field $k$ $($hence we denote it simply by
$\cHT_n$$)$.
\end{corollary}

\begin{proof}
The assertion follows from Proposition~\ref{easy_properties} ii)
and the analogous property of $\cT_n(k)$.
\end{proof}

\begin{corollary}\label{cor:all-the-same-subset}
For every algebraically closed field $k$, we have
$$\cHT_n(k)=\cHTpol_n(k)=\cHTser_n(k).$$
\end{corollary}

\begin{proof}
The fact that $\cHT_n(k)=\cHTpol_n(k)$ follows from
Proposition~\ref{easy_properties}~ii). The inclusion
$\cHTser_n(k)\subseteq\cHT_n(k)$
follows from Propositions~\ref{prop:all-the-same-subset}
and~\ref{easy_properties}~ii), and the reverse
inclusion follows from Proposition~\ref{completion}.
\end{proof}

In light of the above two corollaries, we simply write $\cHT_n$ for
either of the sets $\cHT_n(k)$, $\cHTpol_n(k)$ or $\cHTser_n(k)$.
Note that we have $\cHT_n=\cT_n\cap [0,1]$.

\begin{remark}\label{rem3_1}
We also have $\cHTiso_n(k)=\cTiso_n(k)\cap [0,1)$ for $n\geq 2$.
Indeed, note first that for every locally principal ideal sheaf $\fra$, we
have $\cJ(\fra)=\fra$, hence $1\not\in \cHTiso_n(k)$. On the other
hand, suppose that we have $\lct(\fra)< 1$, where $\fra\subset\cO_X$ is a
proper ideal sheaf, with $X$ affine, nonsingular and $n$-dimensional, and let
$f\in\fra$ be a general linear combination of a system of generators of $\fra$. Since $\cJ(\fra^{\lambda})=\cJ(f^{\lambda})$
for every $\lambda< 1$, we see that $\lct(f)=\lct(\fra)$ and
$\LCT(f)=\LCT(\fra)$.

Together with Proposition~\ref{union_of_is}, this implies that
$\cHTiso_n(k)$ is independent of $k$, and therefore we simply write
$\cHTiso_n$. We also deduce that $\cHT_n=\cHT_{n-1}\cup\cHTiso_n$
for every $n\geq 1$.
\end{remark}

\section{Limits of log canonical thresholds via ultrafilter constructions}

Our main goal in this section is to prove Theorems~\ref{thm_main1}
and \ref{thm_main2}. Recall first some terminology. If $T$ is a
subset of $\R$, then an element $r \in \R$ is said to be a point of
accumulation (or {\it accumulation point}) of $T$ if $T \cap (r -
\epsilon,r + \epsilon) \setminus \{r\} \ne \emptyset$ for every
$\epsilon > 0$. We say that $r$ is a point of accumulation {\it from
above} (resp., {\it from below}) of $T$ if $T \cap (r,r+\epsilon)
\ne \emptyset$ (resp., $T \cap (r - \epsilon,r) \ne \emptyset$) for
every $\epsilon > 0$.

The key arguments in this section are based on ultrafilter 
constructions.
For basic definitions and properties, we refer to \cite{Gol}.
In the following we fix a non-principal ultrafilter $\cU$ on the set of nonnegative
integers $\NN$. We say that a property
${\mathcal P}(m)$ holds for \emph{almost all} $m\in\N$ if the set
$\{m\in\N\mid {\mathcal P}(m)\,\text{holds}\}$ belongs to $\cU$.

The {\it ultraproduct} (with respect to the ultrafilter $\cU$)
of a sequence of sets $\{A_m\}_{m\in\NN}$ will be denoted by
$[A_m]$. Recall that this consists of equivalence classes of
sequences $(a_m)_m$, where $(a_m)_m\sim (b_m)_m$ if
$a_m=b_m$ for almost all $m$. The 
class of a sequence of elements $(a_m)_m$ in $[A_m]$
will be denoted by $[a_m]$.

Given a sequence of functions $f_m\colon A_m \to B_m$, we
denote by $[f_m]\colon [A_m] \to[] [B_m]$ the function that takes $[a_m]$ to
$[f(a_m)]$.
If $A_m=A$ for all $m$, then the corresponding ultraproduct is
simply denoted by $\ltightstar{A}$, and called the
\emph{non-standard extension} of $A$. Note that there is an injective map
$A\hookrightarrow \ltightstar{A}$ that takes $a$ to the class of
$(a,a,\ldots)$. Similarly, a function $u\colon A\to B$ has a
non-standard extension $\ltightstar{u}\colon \ltightstar{A}\to
\ltightstar{B}$.
If $(A_m)_m$ is a sequence of sets, and if $B_m\subseteq A_m$ for
every $m$ (in fact, it is enough to have this inclusion for almost
all $m$), then $[B_m]$ can be considered a subset of $[A_m]$. The
subsets of $[A_m]$ that arise in this way are called
\emph{internal}. Similarly, if $(A_m)_m$ and $(B_m)_m$ are sequences
of sets, then an \emph{internal function} $f\colon
[A_m]\longrightarrow  [B_m]$ is a function of the form $f=[f_m]$ for
suitable $f_m\colon A_m\to B_m$.

As a general principle one observes that if $A$ has
an algebraic structure, then $\ltightstar{A}$ has a similar
structure, too. For example, $\lstar{\RR}$ is an ordered field, and if $k$ is an
algebraically closed field, then $\ltightstar{k}$ is an
algebraically closed field, too. The operations are defined
component-wise, for example $[a_m]+[b_m] =[a_m+b_m]$.

We now turn to the case that will be of particular interest to us.
Suppose that $k$ is a field, and that we have a sequence of
polynomials $f_m\in k[x_1,\ldots,x_n]$. We can view any polynomial
$g \in k[x_1,\ldots,x_n]$ as a function $\N^n \to k$ given by sending the tuple
$(m_1,\ldots,m_n)$ to the coefficient of the monomial
$x_1^{m_1}\cdots x_n^{m_n}$ in $g$. The sequence $(f_m)_m$ gives an
\emph{internal polynomial} $F=[f_m]\in
\ltightstar{(k[x_1,\ldots,x_n])}$, that we can view as a function
$\ltightstar{(\N^n)}\to \ltightstar{k}$. Since we have a natural
inclusion $\N^n\subset\ltightstar{(\N^n)}$, we can restrict $F$ to
$\N^n$ to get a formal power series $\~f\in \ltightstar{k}\llbracket
x_1,\ldots,x_n \rrbracket$. Hence we have the following natural maps
$$
    k[x_1,\ldots,x_n]  \hookrightarrow {\ltightstar{(k[x_1,\ldots,x_n])}}
     \overset{\rho}\to {\ltightstar{k}\llbracket
     x_1,\ldots,x_n\rrbracket}.
$$
Note that we have a natural inclusion
$\ltightstar{k}[x_1,\ldots,x_n]\subset\ltightstar{(k[x_1,\ldots,x_n])}$
such that the restriction of $\rho$ to
$\ltightstar{k}[x_1,\ldots,x_n]$ is the usual inclusion of the
polynomial ring in the formal power series ring. We also remark
that if $f_m(0)=0$ for almost all $m$, then $\~f$ lies in the maximal
ideal, i.e. $\~f(0)=0$.

The above  construction can be generalized to ideals. Given a
sequence of ideals $\fra_m \subseteq k[x_1,\dots,x_n]$, we get the
\emph{internal ideal} $A = [\fra_m] \subseteq
\ltightstar(k[x_1,\dots,x_n])$, with
$$
A=\{F=[f_m]\mid \text{$f_m\in\fra_m$ for almost all $m$} \}.
$$
Note that $A$ is indeed an ideal in the ring of internal polynomials
$\ltightstar{(k[x_1,\ldots,x_n])}$. We denote by $\~\fra \subseteq
\*k\[x_1,\dots,x_n\]$ the ideal generated by $\rho(A)$, that is, the
ideal generated by the restrictions of the elements in $A$ to
$\N^n$. We will refer to $\~\fra$ as \emph{the ideal of power series
associated to} $A$.

We say that an internal polynomial $G = [g_m] \in
\*(k[x_1,\dots,x_n])$ has {\it bounded degree} if there exists an
integer $d$ such that $\deg(g_m)\leq d$ for almost all $m$. Note
that every internal polynomial $G$ of bounded degree can be
represented as $G = [g_m']$, where $g_m'$ are polynomials of degree
uniformly bounded by some $d$, and thus it can be viewed as a
function $\ltightstar{(\N^n)} \to \*k$ that is zero away from a
finite subset of $\N^n$. As such, $G$ can be naturally identified
with a polynomial $g \in \*k[x_1,\dots,x_n]$ (of course, $g$ is
independent of the choice of the polynomials $g_m'$ chosen to
represent $G$).

An internal ideal $B = [\frb_m] \subseteq \*(k[x_1,\dots,x_n])$ is
\emph{generated in degree} $\leq d\in\NN$ if $\frb_m$ can be
generated in degree $\leq d$ for almost all $m$. Since
$\dim_kk[x_1,\ldots,x_n]_{\leq d}={{d+n}\choose {n}} = r+1$, we deduce
that if $B$ is generated in degree $\leq d$, then we may assume that
every $\frb_m$ is generated by a set of $r+1$ polynomials of degree
$\le d$. Thus we can write $\frb_m = (g_{0,m},\dots,g_{r,m})$, with
$\deg g_{i,m} \le d$. Then we see that $B$ is generated by the internal
polynomials $G_i = [g_{i,m}]$ (for $0 \le i \le r$), each having
degree $\leq d$. It follows that the ideal $\~\frb$ of power series
associated to $B$ is the extension of the ideal $\frb\subseteq
\ltightstar{k}[x_1,\ldots,x_n]$ generated by the polynomials in the ring
$\ltightstar{k}[x_1,\ldots,x_n]$ naturally identified with
the $G_i$. We will refer to the ideal $\frb$
(which is independent of the choice of $g_{i,m}$), as \emph{the
ideal of polynomials associated to} $B$.

This applies, for instance, in the following situation. Consider the
maximal ideal $\frm = (x_1,\dots,x_n)$ in the polynomial ring
$k[x_1,\dots,x_n]$. If no confusion is likely to arise, we will also
denote by $\frm$ the ideal generated by the variables in
$\*k[x_1,\dots,x_n]$, and we will write $\~\frm$ for the maximal ideal
of $\*k\[x_1,\ldots,x_n\]$. Note that the internal ideal $M=[\frm]$
is equal to the ideal $(x_1,\ldots,x_n)\cdot \*(k[x_1,\ldots,x_n])$
generated by $x_1,\ldots,x_n$ in $\*(k[x_1,\ldots,x_n])$.

For every internal ideal $A = [\fra_m] \subseteq
\*(k[x_1,\dots,x_n])$ and every $d \in \N$, we consider the internal
ideal
$$
A + M^d := [\fra_m + \frm^d] \subseteq \*(k[x_1,\dots,x_n]).
$$
Notice that $A + M^d$ is generated in degree $\le d$.
We denote by $\fra + \frm^d \subseteq
\*k[x_1,\dots,x_n]$ the ideal of polynomials associated to $A + M^d$
(note that in this notation $\fra$ alone might not be defined
as an ideal).
If $\~\fra$ and $\~\frm^d$ are the ideals in
$\*k\[x_1,\dots,x_n\]$ associated, respectively, to $A$ and to
$M^d$, then $(\fra + \frm^d) \. \*k\[x_1,\dots,x_n\] = \~\fra +
\~\frm^d$.

We now consider the behavior of codimension under the previous
construction. Suppose that $\frb_m \subseteq k[x_1,\dots,x_n]$, for
$m \in \N$, is an ideal generated in degree $\le d$, and let $\frb
\subseteq \*k[x_1,\dots,x_n]$ be the ideal of polynomials associated
to $[\frb_m]$. The codimension of the ideals $\frb_m$ can take only
finitely many values, hence there is a unique integer $e$ such that
$\codim(\frb_m)=e$ for almost all $m$.

\begin{proposition}\label{prop:transfer-of-codimension}
With the above notation, we have $\codim(\frb) = e$.
\end{proposition}

\begin{proof}
We give two different proofs.
Suppose first that $\fra=(f_0,\ldots,f_r)$ is an ideal in
$k[x_1,\ldots,x_n]$ generated by $r+1$ polynomials of degree $\leq
d$, where as above $r={{n+d}\choose n}-1$. It is known that there is
a first-order formula $(\verb"Codim=e")_d$ in the coefficients of
$f_0,\ldots,f_r$ such that for every field $k$ and every
$f_0,\ldots,f_r$ as above, we have $\codim(\fra)=e$ if and only if
the coefficients of the $f_i$'s satisfy this formula
 over $k$
(see Prop. 5.1 in \cite{Sch}). Recall that a first-order formula (in
the free variables ${\bf z}=z_1,\ldots,z_N$) is an expression of the
form
$$\phi=(\exists{{\bf y}_0})(\forall{{\bf y}_1})\cdots
(\exists{{\bf y}_{s-1}})(\forall{{\bf y}_s})\bigvee_{i<m_1}\bigwedge_{j<m_2}
p_{ij}({\bf z},{\bf y})=0\wedge q_{ij}({\bf z},{\bf y})\neq 0,$$ where
$p_{ij}$ and $q_{ij}$ are polynomials over $\ZZ$, and the ${\bf y}_j$
are (possibly empty) tuples of variables. One says that such a
formula is satisfied over $k$ (for given $z_1,\ldots,z_N\in k$) if
we have a true statement when the ${\bf y}$ variables are also
assumed to take values in $k$.

Suppose now that we write our ideals as
$\frb_m=(f_{m,0},\ldots,f_{m,r})\subseteq k[x_1,\ldots,x_n]$, such
that all $f_{m,i}$ have degree $\leq d$. If $G_i=[f_{m,i}]\in
\*k[x_1,\ldots,x_n]$, then we have $\frb=(G_0,\ldots,G_r)$. We see
that $\codim(G_0,\ldots,G_r)=e$ if and only if the coefficients of
the $G_i$ satisfy the formula $(\verb"Codim=e")_d$ over $\*k$.
However, this is the case if and only if for almost all $m\in\NN$,
the coefficients of the $f_{m,i}$ satisfy the formula
$({\verb"Codim=e"})_d$ over $k$. This holds since by assumption we
have $\codim(\frb_m)=e$ for almost all $m$.

For the benefit of the reader without experience in model theory, we give a second,
more transparent proof. After replacing $k$ by a suitable extension, we may assume that 
$k$ is algebraically closed. 
We use the fact that a polynomial ring over a field is Cohen-Macaulay, hence the codimension of an ideal $\fra$ in such a ring is equal to the length of any maximal regular sequence contained in 
$\fra$. Furthermore, if $\codim(\fra)=e$ and $\fra$ is generated by elements of degree $\leq d$,
then any $e$ general linear combinations of these generators with coefficients in $k$ form a regular sequence. In particular,
we get such a regular sequence whose elements
have all degree $\leq d$.

The key ingredients in this proof are two results from \cite{vdDS}. The first one 
(Theorem~1.11 in \emph{loc. cit.}) says that given $d$ and $n$, 
there is a bound
$N=N(d,n)$ such that for every field $K$, and every $g_0,g_1,\ldots,g_{\ell}
\in K[x_1,\ldots,x_n]$, with $\deg(g_i)\leq d$ for every $i$, if $g_0\in (g_1,\ldots,g_{\ell})$, then
$g_0=\sum_{i=1}^{\ell}p_ig_i$, with $\deg(p_i)\leq N$ for every $i$.
The second result (Theorem 1.4 in \emph{loc. cit.}) has a similar flavor:
given $d$ and $n$, 
there is a bound $N'=N'(d,n)$ such that for every field $K$, 
and every
$g_0,g_1,\ldots,g_{\ell}\in K[x_1,\ldots,x_n]$ with $\deg(g_i)\leq d$ for all $i$, the following submodule of
$(K[x_1,\ldots,x_n])^{\ell+1}$
$$\{(q_0,q_1,\ldots,q_{\ell})\in (K[x_1,\ldots,x_n])^{\ell+1}\mid \sum_{i=0}^{\ell}q_ig_i=0\}$$
 is generated by elements
 $(q_0,\ldots,q_{\ell})$ with $\deg(q_i)\leq N'$ for $0\leq i\leq \ell$.
In particular, if $g_0$ is a zero-divisor modulo $(g_1,\ldots,g_{\ell})$, then there are
$q_0,q_1,\ldots,q_{\ell}\in K[x_1,\ldots,x_n]$ of degree $\leq N'$ such that
$q_0g_0=\sum_{i=1}^{\ell}q_ig_i$, and $q_0\not\in (g_1,\ldots,g_{\ell})$.
 It is worth mentioning that
the proofs of these two results use in an essential way non-standard arguments.

Suppose now that we have a sequence of ideals $\frb_m$ as in the statement of the proposition.
For every $m$ in the set $I:=\{m\in\NN\mid\codim(\frb_m)=e\}$, let us choose 
$g_{m,1},\ldots,g_{m,e}\in\frb_m$ of degree $\leq d$, forming a regular sequence. 
We get $g_j=[g_{m,j}]\in\frb$ for $1\leq j\leq e$ (we can take arbitrary $g_{m,j}$
for $m\not\in I$, as these are irrelevant).
We start by showing that
$g_1,\ldots,g_e$ is a regular sequence in $\ltightstar{k}[x_1,\ldots,x_n]$. Indeed, otherwise 
we can find
$i\leq e$ and $h\in\ltightstar{k}[x_1,\ldots,x_n]\smallsetminus (g_1,\ldots,g_{i-1})$ such that
$g_ih\in (g_1,\ldots,g_{i-1})$. This implies that for almost all $m$ we have $g_{m,i}h_m\in (g_{m,1},\ldots,g_{m,i-1})$.
For almost all $m$, we know that $g_{m,1},\ldots,g_{m,e}$ forms a regular sequence, hence $h_m\in (g_{m,1},\ldots,g_{m,i-1})$. Since the degrees of $h_m$ are also bounded 
above by $\deg(h)$, it follows from the first result we quoted from \cite{vdDS} that there is $N$
such that $h_m=\sum_{j=1}^{i-1}p_{m,j}g_{m,j}$ and $\deg(p_{m,j})\leq N$ for almost all $m$. 
In this case $p_j=[p_{m,j}]\in \ltightstar{k}[x_1,\ldots,x_n]$  and 
$h=\sum_{j=1}^{i-1}p_jg_j$, a contradiction. 

In order to conclude the proof, it is enough to also show that $g_1,\ldots,g_e$ 
is a maximal regular sequence in the ideal $\frb$. For this, suppose that
there is $h\in\frb$  such that $g_1,\ldots,g_e,h$ is a regular sequence. 
For almost all $m$, we have $\codim(\frb_m)=e$, hence $h_m$ is a zero-divisor
modulo $(g_{m,1},\ldots,g_{m,e})$. By the second result we quoted from \cite{vdDS}
we deduce that we can find $N'$ and polynomials $q_{m,0},\ldots,q_{m,e}\in k[x_1,\ldots,x_n]$ 
of degree $\leq N'$
such that $q_{m,0}\not\in (g_{m,1},\ldots,g_{m,e})$ and 
$q_{m,0}h_m=\sum_{i=1}^eq_{m,i}g_{m,i}$ for almost all $m$. If we put $q_i=[q_{m,i}]$,
then $q_0h=\sum_{i=1}^eq_ig_i$, and our assumption on $h$ implies 
$q_0\in (g_1,\ldots,g_e)$. This contradicts the fact that $q_{m,0}\not\in (g_{m,1},\ldots,g_{m,e})$
for almost all $m$,
and completes our second proof.
\end{proof}

We now turn to log canonical thresholds. From now on $k$ is a fixed
algebraically closed field of characteristic zero. The function
$$
\lct_0\colon\{\fra\subseteq k[x_1,\ldots,x_n]\mid
\text{$\fra$ ideal, $\fra\subseteq (x_1,\ldots,x_n)$} \}\to\R
$$
extends to a function $\*\lct_0$, defined for internal ideals contained
in $(x_1,\ldots,x_n)\cdot\*(k[x_1,\ldots,x_n])$ and taking values in
$\lstar{\R}$. Explicitly, we have
$$
\*\lct_0([\fra_m])=[\lct_0(\fra_m)].
$$

Recall that for every bounded $u\in \lstar{\R}$ there is a unique real
number $\sh(u)$, the \emph{shadow} of $u$, characterized by
$|u-\sh(u)|<\epsilon$ for every positive real number $\epsilon$ (we
abuse the notation by writing $|u|$, instead of $\*|u|$ for
$u\in\lstar\R$). For a discussion of this notion, we refer to
\cite{Gol}, \S 5.6. A useful fact (see Theorem~6.1 in \emph{loc.
cit}.) implies that if $\{c_m\}_m$ is a sequence of real numbers
converging to $c$, then ${\rm sh}([c_m])=c$. The following is the
key result that allows us to interpret limits of log canonical
thresholds as log canonical thresholds of ideals of formal power
series.

\begin{proposition}\label{prop:lct=shadow-of-star-lct}
If $A = [\fra_m] \subseteq \*(k[x_1,\dots,x_n])$ is an internal
ideal contained in $(x_1,\ldots,x_n)$, and if $\~\fra \subseteq
\*k\[x_1,\dots,x_n\]$ is the ideal of power series associated to
$A$, then
$$
\sh\big(\*\lct_0(A)\big) = \lct(\~\fra).
$$
\end{proposition}

We first prove the following lemma.

\begin{lemma}\label{lem:lct=star-lct:bounded-case}
If $B = [\frb_m]\subseteq \*(k[x_1,\ldots,x_n])$ is an internal
ideal contained in $(x_1,\ldots,x_n)\cdot\*k[x_1,\ldots,x_n]$ and
generated in degree $\le d$, and if $\frb \subseteq \*k[x_1, \dots,
x_n]$ is the ideal of polynomials associated to $B$, then
$$
\*\lct_0(B) = \lct_0(\frb)
$$
$($where $\R$ is considered as a subset of $\*\R$ in the usual
way$)$.
\end{lemma}

\begin{proof}
We can assume that each $\frb_m$ is contained in $\frm$ and is
generated in degree $\le d$. We have seen in the previous section
that there is an affine space $W={\mathbb A}^{r^2}$, where $r =
\binom{d+n}{n} - 1$, such that for every  field extension
$k\subseteq L$, with $L$ algebraically closed, we have a functorial
surjective map
$$
\phi_{L}\colon W(L)\to\{\frb\subseteq (x_1,\ldots,x_n)
L[x_1,\ldots,x_n]\mid \text{$\frb$ generated in degree $\le d$} \}.
$$
Moreover, we have seen that $W$ can be written as the disjoint union
of finitely many locally closed subsets $W_{\lambda_i}$ (with
$\lambda_i\in\RR$) defined over $k$, such that for every $L$
$$
W_{\lambda_i}(L)=\{u\in W(L)\mid \lct_0(\phi_L(u))=\lambda_i\}.
$$
Therefore there is $\lambda$ such that $\frb_m\in W_{\lambda}$ for
almost all $m$, hence $\*\lct_0(B)=\lambda$.

If $L$ is as above, and if $\fra\subseteq
(x_1,\ldots,x_n)L[x_1,\ldots,x_n]$, let us denote by $\fra^{\langle
i\rangle}$ the ideal defining the $i^{\rm th}$ jet scheme of the
subscheme of $\AAA_L^n$ defined by $\fra$. This is an ideal in the
polynomial ring $R_i:=L[x_j^{(\ell)}|1 \le j\leq n, 0\leq \ell\leq i]$
(where we identify $x_j^{(0)}$ with $x_j$). If $\fra$ is generated
by $\{f_{\alpha}\}_{\alpha}$, then $\fra^{\langle i\rangle}$ is
generated by
$\{f_{\alpha},f_{\alpha}',\ldots,f_{\alpha}^{(i)}\}_{\alpha}$, where
$f^{(q)}=D^q(f)$, with $D$ being a derivation of $R_i$ taking
$x_j^{(\ell)}$ to $x_j^{(\ell+1)}$. We refer to \cite{EM}, \S 3 for
this description by equations of the jet schemes. We denote by
$\overline{\fra^{\langle i\rangle}}\subseteq R_i/(x_1,\ldots,x_n)$
the ideal defining the fiber over $0$ in the above jet scheme.

It follows from the above description that in our
setup the ideal
$$\overline{\frb^{\langle
i\rangle}}\subseteq\*k[x_j^{(1)},\ldots,x_j^{(i)}|1\leq j\leq n]$$
is the ideal of polynomials associated to the internal ideal
$[\overline{\frb_m^{\langle i\rangle}}]$. Note that this is an
internal ideal generated in degree $\leq d$. In particular, it
follows from Proposition~\ref{prop:transfer-of-codimension} that
$\codim(\overline{\frb^{\langle i\rangle}})=e_i$, where $e_i$ is the
unique integer such that $\codim(\overline{\frb_m^{\langle
i\rangle}})=e_i$ for almost all $m$.

On the other hand, we have seen in \S 3 that there is
a positive integer $N$ such that
for every algebraically closed field extension $L$ of $k$, and every ideal
$\fra\subseteq (x_1,\ldots,x_n)L[x_1,\ldots,x_n]$ that is generated in
degree $\le d$, we have
$\lct_0(\fra)=\lambda$ if and only if
$\codim(\overline{\fra^{\langle i\rangle}})$ satisfies suitable
inequalities for every $i\leq N$.
Indeed, note that we have only finitely many possible log canonical thresholds for these ideals, hence 
given such a log canonical threshold $\lambda$, we can find $\lambda'>\lambda$
such that there is no possible log canonical threshold in $(\lambda,\lambda')$. 
We can then choose $N$ such that
the condition $\lct_0(\fra) \ge \lambda$ is equivalent 
to $\codim(\overline{\fra^{\langle i\rangle}})\geq (n-\lambda)(i+1)$ for every $i \le N$,
whereas the condition $\lct_0(\fra) \le \lambda$ is equivalent to
$\codim(\overline{\fra^{\langle i\rangle}})\leq (n-\lambda')(i+1)$ for some $i \le N$.

It follows that if these inequalities are satisfied by
$\codim(\overline{\frb_m^{\langle i\rangle}})$ for almost all $m$, then they are
satisfied also by $\codim(\overline{\frb^{\langle i\rangle}})$.
We conclude that $\lct_0(\frb)=\lambda$, as required.
\end{proof}

\begin{proof}[Proof of Proposition~\ref{prop:lct=shadow-of-star-lct}]
For every fixed $d\geq 1$, we consider the internal ideal $A + M^d
= [\fra_m + \frm^d]$. Note that
this internal ideal is generated in degree $\le d$. Let
$$
\fra + \frm^d \subseteq \*k[x_1,\dots,x_n] \and \~\fra + \~\frm^d
\subseteq \*k\[x_1,\dots,x_n\]
$$
denote, respectively, the ideal of polynomials and the ideal of
power series that are associated to $A + M^d$. Since $\~\fra +
\~\frm^d = (\fra + \frm^d)\. \*k\[x_1,\dots,x_n\]$, we have
$$
\lct_0(\fra + \frm^d) = \lct(\~\fra + \~\frm^d)
$$
by Proposition~\ref{completion}. Applying Corollary~\ref{limit2}, we
get
$$
\big|\lct(\~\fra) - \lct(\~\fra + \~\frm^d)\big| \le n/d
$$
for every $d\geq 1$, and similarly
$$
\big|\*\lct_0(A) - \*\lct_0(A + M^d)\big| = \left[|\lct_0(\fra_m) -
\lct_0(\fra_m + \frm^d)|\right] \le n/d.
$$
On the other hand, $\*\lct_0(A + M^d) = \lct_0(\fra + \frm^d)$ for
any $d$ by Lemma~\ref{lem:lct=star-lct:bounded-case}. Therefore
$$
\big|\lct(\~\fra) - \*\lct_0(A)\big| \le 2n/d.
$$
As this holds for every $d$, we get the assertion in the
proposition.
\end{proof}

\begin{remark}
One can carry a construction analogous to the one considered in this section when
starting with a sequence
of power series $h_m\in k\llbracket x_1,\ldots,x_n \rrbracket$
rather than a sequence of polynomials.
This produces an {\it internal power series}
$H=[h_m]\in \ltightstar{(k\llbracket x_1,\ldots,x_n \rrbracket)}$,
and thus, after truncation of the unbounded degree terms,
a power series $\~h\in \ltightstar{k}\llbracket
x_1,\ldots,x_n \rrbracket$. Similarly, starting with a sequence of ideals
$\fra_m\subseteq k\llbracket x_1,\ldots,x_n\rrbracket$, we get an ideal 
$A$ in $\ltightstar(k\llbracket x_1,\ldots,x_n\rrbracket)$, and after truncation,
an ideal $\widetilde{\fra}\subseteq \ltightstar{k}\llbracket x_1,\ldots,x_n\rrbracket$.
All previous results have analogues in this setting, the proofs being the same. For example,
we again have ${\rm sh}(\ltightstar{\lct}_0(A))=\lct(\widetilde{\fra})$. 
\end{remark}

We are now ready to prove our main results stated in the
Introduction.

\begin{proof}[Proof of Theorem~\ref{thm_main1}]
Consider a sequence $\{c_m\}_m$ with $c_m\in\cT_n$ for all $m$, and
with $\lim_{m\to\infty}c_m=c$. Fix an algebraically closed field $k$
of characteristic zero. By
Proposition~\ref{prop:all-the-same-subset} we can find ideals
$\fra_m\subseteq (x_1,\ldots,x_n)k[x_1,\ldots,x_n]$ such that
$c_m=\lct_0(\fra_m)$. Let $A=[\fra_m]$, and let $\~\fra\subseteq
\*k\[x_1,\ldots,x_n\]$ be the ideal of formal power series
associated to $A$. It follows from
Proposition~\ref{prop:lct=shadow-of-star-lct} that $${\rm
sh}([c_m])={\rm sh}(\*\lct_0(A))=\lct(\~\fra).$$ On the other hand,
since $\{c_m\}_m$ converges to $c$, we have ${\rm sh}([c_m])=c$. We
conclude that $c\in\cTser_n$, and therefore $c\in\cT_n$, by
Proposition~\ref{prop:all-the-same-subset}.
\end{proof}

\begin{proof}[Proof of Theorem~\ref{thm_main2}]
The fact that every element of $\cT_{n-1}$ is a point of
accumulation from above for $\cT_n$ is well-known. It follows from
the fact that if $c=\lct(\fra)$, for some ideal sheaf $\fra$ on
a nonsingular $(n-1)$-dimensional
variety $X$, then $c+\frac{1}{m}=\lct(\fra+(t^m))$,
where $\fra+(t^m)$ is an ideal on $X\times{\mathbb A}^1$, $t$
denoting the coordinate on ${\mathbb A}^1$
(see, for example, Proposition~1.20 in \cite{ELSV}).

Suppose now that we have a strictly decreasing sequence $\{c_m\}_{m
\in \N}$ in $\cT_n$, and let $c= \lim_{m \to \infty} c_m$.
Fix an algebraically closed field $k$ of
characteristic zero. By Proposition~\ref{prop:all-the-same-subset}
we can find ideals $\fra_m\subseteq
(x_1,\ldots,x_m)k[x_1,\ldots,x_n]$ such that $c_m=\lct_0(\fra_m)$.
Let $A=[\fra_m]$, and let $\~\fra\subseteq \*k\[x_1,\ldots,x_n\]$ be
the ideal of formal power series associated to $A$. As in the proof
of Theorem~\ref{thm_main1}, we deduce that $c=\lct(\~\fra)$.

As usual, we put $\frm=(x_1,\ldots,x_n)\.k[x_1,\ldots,x_n]$
and $\~\frm = (x_1,\ldots,x_n)\.\*k\[x_1,\ldots,x_n\]$.
In order to complete the proof, it is enough to show that
\begin{equation}\label{enough}
\lct\left(\~\fra+\~\frm^d\right)>\lct(\~\fra)
\end{equation}
 for every
$d\geq 1$. Indeed, this implies by Proposition~\ref{limit1} that the
locus $\LCT(\~\fra)$ is positive dimensional, hence $c\in\cT_{n-1}$
by Proposition~\ref{reduction_to_zero-dimensional}.

 Fix now $d\geq 1$,  and consider the ideals $\fra_m + \frm^d
\subseteq k[x_1,\ldots,x_n]$. Note that the internal ideal
$[\fra_m+\frm^d]$ is equal to $A+M^d$, where
$M=(x_1,\ldots,x_n)\cdot\*(k[x_1,\ldots,x_n])$. Therefore the ideal
of formal power series associated to $[\fra_m+\frm^d]$ is $\~\fra+
\~\frm^d$. Since the internal ideal
$[\fra_m+\frm^d]$ is generated in degree $\leq d$, it follows from
Proposition~\ref{lem:finitely-many-lct} that there is
$\lambda_d\in\RR$ such that $\lct_0(\fra_m+\frm^d)=\lambda_d$ for
almost all $m$, and hence we obtain
$\lct\left(\~\fra+\~\frm^d\right)=\lambda_d$
by Proposition~\ref{prop:lct=shadow-of-star-lct}.

On the other hand, for every $m$ we have
$\lct_0(\fra_m+\frm^d)\geq\lct_0(\fra_m)=c_m>c$, hence
$\lambda_d>c$. This shows (\ref{enough}) and completes the proof of
the theorem.
\end{proof}

\begin{remark}\label{hypersurface_case}
Since $\cHT_n=\cT_n\cap [0,1]$ we immediately obtain variants of
Theorems~\ref{thm_main1} and \ref{thm_main2} for log canonical
thresholds of hypersurfaces: each set $\cHT_n$ is closed, and the
set of accumulation points from above of $\cHT_n$ is
$\cHT_{n-1}\smallsetminus\{1\}$.
\end{remark}

\section{Comments on the ACC Conjecture}

By Proposition~\ref{easy_properties}, we have
$\cHT_n\subseteq\cT_n\subseteq n\cdot\cHT_n$. This implies that
Conjecture~\ref{conj_shokurov} holds for $n$ if and only if $\cHT_n$
has no points of accumulation from below. We now show that this
holds for every $n$ if and only if for every $n$ there is no
strictly increasing sequence in $\cHT_n$ converging to $1$.

\begin{proof}[Proof of Corollary~\ref{cor_thm_main1}]
Fix  an algebraically closed field $k$ of characteristic zero.
Suppose that we have a strictly increasing sequence $\{c_m\}_m$ in
$\cHT_n(k)$. By Corollary~\ref{cor:all-the-same-subset}, we may
write $c_m=\lct_0(f_m)$ for some $f_m\in k[x_1,\ldots,x_n]$ with
$f_m(0)=0$. Theorem~\ref{thm_main1} (see also
Remark~\ref{hypersurface_case}) gives
$c:=\lim_{m\to\infty}c_m\in\cHT_n(k)$. In particular $c\in\Q$, hence
we may write $c=\frac{a}{b}$ for positive integers $a$ and $b$. Note
that $a\leq b$, and let $N=n+b-a$ and
$$
g_m=f_m+\sum_{i=1}^{b-a}y_i^b\in
k[x_1,\ldots,x_n,y_1,\ldots,y_{b-a}].
$$
A special case of the Thom-Sebastiani Theorem (see Proposition~8.21
in \cite{kollar}) implies that
$$\lct_0(g_m)=\min\left\{1,\lct_0(f_m)+
\sum_{i=1}^{b-a}\frac{1}{b}\right\}=\lct_0(f_m)+\frac{b-a}{b},$$
hence $\{\lct(g_m)\}_m$ is a strictly increasing sequence in
$\cHT_N(k)$ converging to $1$.
\end{proof}

\bigskip

We now turn to the proof of Proposition~\ref{prop:intro} from the
Introduction. In fact, we will give a stronger statement in
Proposition~\ref{equivalence} below, that also interprets the ACC
Conjecture as a semicontinuity property of log canonical thresholds
of formal power series.

From now on we fix an uncountable algebraically closed field $k$ of
characteristic zero. We consider the set $R=k\[x_1,\ldots,x_n\]$ as
the set of $k$-valued points of an infinite-dimensional affine space
over $k$, parametrizing the coefficients of the power series. As
such, it carries a natural Zariski topology.

This can be described, more precisely, as follows. If $\frm$ denotes
the maximal ideal of $R$, then
$R\simeq\projlim_{\ell}R/\frm^{\ell}$. Each $R/\frm^{\ell}$
parametrizes polynomials over $k$ of degree $<\ell$, and as such it
consists of the $k$-points of an affine space of dimension
${{n+\ell-1}\choose {n}}$. The Zariski topology on $R$ is the
projective limit of the Zariski topologies on each of the
$R/\frm^{\ell}$. Note that $\frm\subset R$ is a closed subset.

We denote by $\psi_{\ell}\colon R\to R/\frm^{\ell}$ the natural
projection maps. A subset of $R$ is a \emph{cylinder} if it is of
the form $\psi_{\ell}^{-1}(S)$ for some $\ell$ and some subset $S$
of $R/\frm^{\ell}$ (note that we do not put any condition on $S$). A
cylinder $C$ is \emph{constructible}, \emph{open} or \emph{closed}
if $S$ is constructible, open or closed, respectively. Since the
projections $R/\frm^{\ell+1} \to R/\frm^{\ell}$ are continuous and
open (being flat), it follows that these notions are well-defined.
These are variants of the corresponding notions when instead of $R$
one considers
the space of arcs of a smooth variety (see for example \cite{ELM}).

\begin{remark}\label{rem_batyrev}
A variant of an argument due to Batyrev \cite{batyrev} in the case
of spaces of arcs (see also Lemma~1.2 in \cite{ELM}) implies that if
$C_1\supseteq C_2\supseteq\cdots $ is a sequence of constructible
cylinders with $\cap_mC_m=\emptyset$, then $C_m=\emptyset$ for some
$m$. This is the key point where we use the fact that $k$ is
uncountable.
\end{remark}

For every $n \ge 1$ and $c \in \R$, we consider the set
$$
{\mathcal R}_n(c):=\{f\in \frm\mid\lct(f)\geq c\}.
$$

\begin{proposition}\label{equivalence}
For every $n\geq 1$ and $c \in \R$, the following assertions are equivalent:
\begin{enumerate}
\item[i)] $c$ is not a point of accumulation from below of $\cHT_n$.
\item[ii)] ${\mathcal R}_n(c)$ is a cylinder in $R$.
\item[iii)] ${\mathcal R}_n(c)$ is open in $\frm$.
\end{enumerate}
\end{proposition}

\begin{proof}
For every $f\in R$ and every nonnegative integer $d$, we denote by
$f_{\leq d}$ the truncation of $f$ of degree $\leq d$. It is
convenient to consider also the map $\iota_d\colon R/\frm^{d+1}\to
R$ that identifies each coset $h + \frm^{d+1}$ with the corresponding polynomial
$h_{\le d}$ of degree $\leq d$.

In order to prove that i) $\Rightarrow$ ii) note that if $c$ is not
a point of accumulation from below for the set $\cHT_n$, then there
is $\epsilon>0$ such that $(c-\epsilon,c)\cap\cHT_n=\emptyset$. It
follows from Corollary~\ref{limit2} that if
$\frac{n}{d+1}<\frac{\epsilon}{2}$, then $f\in {\mathcal R}_n(c)$ if
and only if $\lct_0(f_{\leq d})>c-\frac{\epsilon}{2}$. This implies
that the condition for having $f\in {\mathcal R}_n(c)$ depends only
on $\psi_{d+1}(f)$. Therefore ${\mathcal
R}_n(c)=\psi_{d+1}^{-1}(\psi_{d+1}({\mathcal R}_n(c)))$,  hence it
is a cylinder.

Conversely, suppose now that we have a strictly increasing sequence
$c_m$ in $\cHT_n$ converging to $c$, and that ${\mathcal R}_n(c)$ is
a cylinder. We note that in general, for every $c' > 0$, we may write
${\mathcal R}_n(c')$ as a countable intersection of constructible
cylinders. Indeed, it follows from Corollary~\ref{limit2} that we
can write
$$
{\mathcal R}_n(c')=\bigcap_{\ell\geq 1}\psi_{\ell}^{-1}(S_{\ell}),
$$
where $S_{\ell}:= \left\{h+ \frm^\ell\in
\frm/\frm^{\ell}\mid\lct_0(h_{\le \ell-1})\geq
c'-\frac{n}{\ell}\right\}$ is open in $\frm/\frm^{\ell}$ by the
Semicontinuity Theorem for log canonical thresholds (see
Corollary~9.5.39 in \cite{positivity}).

Therefore we write
$${\mathcal R}_n(c_m)=\bigcap_{\ell\geq 1}C^m_{\ell},$$
where each $C^m_{\ell}$ is a constructible cylinder. Note that
$\sup_mc_m=c$, hence
$$
{\mathcal R}_n(c)=\bigcap_m{\mathcal R}_n(c_m)=\bigcap_{m,\ell}C^m_{\ell}.
$$
Since ${\mathcal R}_n(c)$ is a cylinder, it follows from
Remark~\ref{rem_batyrev} that ${\mathcal R}_n(c)$ is the
intersection of finitely many $C^m_{\ell}$. In particular,
${\mathcal R}_n(c)={\mathcal R}_n(c_m)$ for $m\gg 0$. However, by
Corollary~\ref{cor:all-the-same-subset} we can find $h_m\in R$ such
that $c_m=\lct(h_m)$. Since $c_m<c$ for every $c$, we have $h_m\in
{\mathcal R}_n(c_m)\smallsetminus {\mathcal R}_n(c)$ for all $m$, a
contradiction. This completes the proof of i)~$\Leftrightarrow$~ii).

Suppose now that $c$ is fixed, and that ${\mathcal R}_n(c)$ is a
cylinder. In this case, we can write ${\mathcal R}_n(c)
=\psi_{\ell}^{-1}(S)$, and $S$ can be identified via
$\iota_{\ell-1}$ with the set of those polynomials $g$ of degree
$\leq \ell-1$ such that $g(0)=0$ and $\lct_0(g)\geq c$. The
Semicontinuity Theorem for log canonical thresholds implies that $S$
is open in $\frm/\frm^{\ell}$, hence ${\mathcal R}_n(c)$ is open in
$\frm$. This shows that ii)~$\Rightarrow$~iii).

Conversely, suppose that ${\mathcal R}_n(c)$ is open in $\frm$.
 Since $\frm\smallsetminus {\mathcal
R}_n(c)$ is closed, it follows from the definition of the Zariski
topology on $R$ that we can write
$$
\frm\smallsetminus {\mathcal R}_n(c)=\bigcap_{\ell\geq
1}\psi_{\ell}^{-1}(Z_{\ell})
$$
for suitable closed subsets $Z_{\ell}\subseteq R/\frm^{\ell}$. We
may clearly assume that
$\psi_{\ell}^{-1}(Z_{\ell})\supseteq\psi_{\ell+1}^{-1}(Z_{\ell+1})$
for every $\ell$. On the other hand, we have seen that we can write
$$
{\mathcal R}_n(c)=\bigcap_{\ell\geq 1}C_{\ell},
$$
where each $C_{\ell}$ is a constructible cylinder, and we may assume
that $C_{\ell}\supseteq C_{\ell+1}$ for every $\ell$. We deduce that
if we put $C'_{\ell}=C_{\ell}\cap\psi_{\ell}^{-1}(Z_{\ell})$, then
$C'_{\ell}\supseteq C'_{\ell+1}$ for every $\ell$, and
$\bigcap_{\ell}C'_{\ell}=\emptyset$. It follows from
Remark~\ref{rem_batyrev} that there is $\ell$ such that
$C'_{\ell}=\emptyset$. In this case we have $C_{\ell}={\mathcal
R}_n(c)=\frm\smallsetminus\psi_{\ell}^{-1}(Z_{\ell})$, hence
${\mathcal R}_n(c)$ is a cylinder. This completes the proof of the
proposition.
\end{proof}

\begin{remark}
Suppose that Conjecture~\ref{conj_shokurov} holds for $n$, and let
$c>0$. It follows from the above proposition that if $k$ is an
uncountable algebraically closed field of characteristic zero, then
we can find $N(n,c)$ such that the condition for $f\in
k\[x_1,\ldots,x_n\]$ with $f(0)=0$ to satisfy $\lct(f)\geq c$
depends only on the truncation of $f$ up to level $N(n,c)$. In fact,
this integer $N(n,c)$ is independent on $k$: we have seen in the
proof of the proposition that it depends only on the largest element
in $\cHT_n$ that is $<c$. Moreover, $N(n,c)$ satisfies the same
property for formal power series over every algebraically closed
field $k$ of characteristic zero, as can be seen by taking an
uncountable extension of $k$.
\end{remark}

\begin{remark}
We deduce from Corollary~\ref{cor1} that in order to prove
Conjecture~\ref{conj_shokurov} for a given $n$ it is enough to show
that every ${\mathcal R}_n(c)$ is a cylinder, for $c\in\QQ$.
Furthermore, it follows from Corollary~\ref{cor_thm_main1} that in
order to prove Conjecture~\ref{conj_shokurov}, it
is enough to show that for every $n$ the set ${\mathcal R}_n(1)$ is a cylinder.
\end{remark}

\providecommand{\bysame}{\leavevmode \hbox \o3em
{\hrulefill}\thinspace}


\end{document}